\documentclass[letterpaper,english]{amsart}
\usepackage[T1]{fontenc}
\usepackage[latin9]{inputenc}
\usepackage{color}
\usepackage{babel}

\usepackage{amsthm}
\usepackage{graphicx}
\usepackage{amssymb}
\usepackage{esint}
\usepackage[unicode=true, 
 bookmarks=true,bookmarksnumbered=false,bookmarksopen=false,
 breaklinks=false,pdfborder={0 0 1},backref=false,colorlinks=false]
 {hyperref}

\makeatletter

\numberwithin{equation}{section}
\numberwithin{figure}{section}
\theoremstyle{plain}
\newtheorem{thm}{Theorem}[section]
  \theoremstyle{definition}
  \newtheorem{defn}[thm]{Definition}
  \theoremstyle{plain}
  \newtheorem{prop}[thm]{Proposition}
 \theoremstyle{definition}
  \newtheorem{example}[thm]{Example}
  \theoremstyle{remark}
  \newtheorem{rem}[thm]{Remark}
  \theoremstyle{plain}
  \newtheorem{cor}[thm]{Corollary}
  \theoremstyle{plain}
  \newtheorem{lem}[thm]{Lemma}
  \theoremstyle{remark}
  \newtheorem*{acknowledgement*}{Acknowledgement}

\newcommand{\gph}{\mbox{\rm gph}}
\newcommand{\dom}{\mbox{\rm dom}}
\newcommand{\conv}{\mbox{\rm conv}}
\newcommand{\lip}{\mbox{\rm lip}}
\newcommand{\cl}{\mbox{\rm cl}}

\newcommand{\calm}{\mbox{\rm clm}}
\title[Generalized differentiation with positively homogeneous maps]{Generalized differentiation with \\positively homogeneous maps: Applications \\in set-valued analysis and metric regularity}

\makeatother

\begin{document}

\author{C.H. Jeffrey Pang}

\curraddr{Massachusetts Institute of Technology, Department of Mathematics,
2-334, 77 Massachusetts Avenue, Cambridge MA 02139-4307.}

\email{chj2pang@mit.edu}

\date{\today}
\begin{abstract}
We propose a new concept of generalized differentiation of set-valued
maps that captures first order information. This concept encompasses
the standard notions of Fr\'{e}chet differentiability, strict differentiability,
calmness and Lipschitz continuity in single-valued maps, and the Aubin
property and Lipschitz continuity in set-valued maps. We present calculus
rules, sharpen the relationship between the Aubin property and coderivatives,
and study how metric regularity and open covering can be refined to
have a directional property similar to our concept of generalized
differentiation. Finally, we discuss the relationship between the
robust form of generalized differentiation and its one sided counterpart.
\end{abstract}



\maketitle
\tableofcontents{}

\section{Introduction}

We say that $S$ is a \emph{set-valued map }(or a \emph{multi-valued
function}, or \emph{multifunction}) if for all $x\in X$, $S(x)$
is a subset of $Y$, and we denote set-valued maps by $S:X\rightrightarrows Y$.
Many problems of feasibility, control, optimality and equilibrium
are set-valued in nature, and are best treated with methods in set-valued
analysis. The texts \cite{AF90,Beer93,KT84} contain much of the theory
of set-valued analysis. Set-valued analysis serves as a foundation
for the theory of differential inclusions \cite{AC84,Aub91}, control
theory \cite{Fran88} and variational analysis \cite{RW98,Mor06,DR09},
which in turn have many applications in applied mathematics. We refer
to these texts for the abundant bibliography on the history of set-valued
analysis.

The main contribution of this paper is to introduce a new concept
of generalized differentiation (Definitions \ref{def:set-valued-diff}
and \ref{def:pseudo-T-diff}) using positively homogeneous maps. Any
reasonable definition of a derivative for set-valued maps has to describe
changes in the set in terms of the input variables. Using the Pompieu-Hausdorff
distance (a metric on the space of nonempty compact sets), one obtains
the classical definition of Lipschitz continuity of set-valued maps.
The concept introduced in this paper provides a more precise tool
that incoporates first order information in a set-valued map, encompassing
the standard notions of Fr\'{e}chet differentiability, strict differentiability,
calmness and Lipschitz continuity in single-valued maps, and the Aubin
property and Lipschitz continuity in set-valued maps. We illustrate
how this new concept relates to, and extends, existing methods in
variational and set-valued analysis. To motivate our discussion,
we revisit the relation between the Clarke subdifferential and Clarke
Jacobian and the nonsmooth behavior of functions that can be traced
back to \cite{Iof81}.

Other than the first four sections which provide the necessary background
for the rest of the paper, the last four sections can be read in any
order. In Section \ref{sec:single-valued-diff}, we recall generalized
differentiation for single-valued functions, which was first proposed
by Ioffe \cite{Ioffe79,Iof81}. We relate the generalized derivative
to common notions in classical and variational analysis, paying particular
attention to the Clarke subdifferential and the Clarke Jacobian. In
 Section \ref{sec:set-valued-diff}, we define generalized differentiation
for set-valued maps, and illustrate the lack of relation between our
generalized derivatives and the notions of set-valued derivatives
based on the tangent cones, namely semidifferentiability \cite{Penot84}
and proto-differentiability \cite{Roc89}. We present calculus rules
in Section \ref{sec:Calculus}.

The Aubin property (see Definition \ref{def:calm-aubin}), which is
commonly attributed to \cite{Aub84}, is a method of analyzing local
Lipschitz continuity of set-valued maps. In Section \ref{sec:Boris-extended},
we revisit the classical relationship between the Aubin property and
the coderivatives of a set-valued map. This relationship is referred
to as the Mordukhovich criterion in \cite{RW98}. Since the coderivatives
of a set-valued map can be calculated in many applications and enjoy
an effective calculus, this relationship is an important tool in the
study of the Lipschitz properties of set-valued maps. We will show
that the coderivatives actually give more information on the local
Lipschitz continuity property in our language of generalized derivatives.

It is well known that the Aubin property is related to metric regularity
and open covering \cite{BZ88,Mor93,Penot89}. Open covering is sometimes
known as linear openness. Metric regularity is important in the analysis
of solutions to $\bar{y}\in S(\bar{x})$, while open covering studies
local covering properties of a set-valued map. Both metric regularity
and open covering can be viewed as a study of set-valued maps whose
inverse has the Aubin property. For more on metric regularity, we
refer the reader to \cite{RW98,Mor06,DR04,Iof00,KK02}. In Section
\ref{sec:extended-metric-regularity}, we take a new look at metric
regularity and open covering in view of our definitions of the generalized
derivatives. We study metric regularity and open covering in a much
broader framework, illustrating that a directional behavior similar
to that in our definition of generalized derivatives is present in
metric regularity and open covering.

In Section \ref{sec:Strict--T-from-T}, we discuss how the (basic
and strict) generalized derivatives defined in Sections \ref{sec:single-valued-diff}
and \ref{sec:set-valued-diff} relate to each other. As particular
cases, we obtain an equivalent criterion for strict differentiability
of set-valued maps, and a relationship between calmness and Lipschitz
continuity in both single-valued and set-valued maps. As far as we
are aware, the relation between calmness and Lipschitz continuity
in set-valued maps was first discussed in \cite{Li94,Rob07}.

\section{Preliminaries and notation}

Throughout this paper, we shall assume that $X$ and $Y$ are Banach
spaces. In most cases, we follow the notation of \cite{RW98}. Given
two sets $A,B\subset X$, the notation $A+B$ stands for the \emph{Minkowski
sum} of two sets, defined by\[
A+B:=\{a+b\mid a\in A,b\in B\}.\]
The notation $A-B$ is interpreted as $A+(-B)$. We use $\left\langle \cdot,\cdot\right\rangle :X^{*}\times X\rightarrow\mathbb{R}$,
where $\left\langle \zeta,x\right\rangle :=\zeta(x)$, to denote the
usual dual relation. In Hilbert spaces (and hence in $\mathbb{R}^{n}$),
$\left\langle \cdot,\cdot\right\rangle $ reduces to the usual inner
product. The notation $x\xrightarrow[D]{}\bar{x}$ means that we take
sequences $\{x_{i}\}_{i=1}^{\infty}\subset D$ such that $\lim_{i\to\infty}x_{i}=\bar{x}$.
The closed ball with center $x$ and radius $r$ is denoted by $\mathbb{B}(x,r)$,
while $\mathbb{B}$ denotes the closed unit ball. 

We say that the set-valued map $S:X\rightrightarrows Y$ is \emph{closed-valued
}if $S(x)$ is closed for all $x\in X$, and it is \emph{convex-valued}
if $S(x)$ is convex for all $x\in X$. A \emph{closed }set-valued
map is a map whose graph is closed. We say that $C\subset X$ is a
\emph{cone }if $\mathbf{0}\in C$ and $\lambda x\in C$ for all $\lambda>0$
and $x\in C$. 

The \emph{graph} of a set-valued map $\gph(S)\subset X\times Y$ is
the set ${\{(x,y)\mid y\in S(x)\}}$. The set-valued map $S^{-1}:Y\rightrightarrows X$
is defined by $S^{-1}(y):=\{x\mid y\in S(x)\}$, and $\gph(S^{-1})=\{(y,x)\mid(x,y)\in\gph(S)\}$.
\begin{defn}
A set-valued map $T:X\rightrightarrows Y$ is \emph{positively homogeneous}
if \[
T(\mathbf{0})\mbox{ is a cone, and }T(kw)=kT(w)\mbox{ for all }k>0\mbox{ and }w\in X.\]

\end{defn}
It is clear that $T$ is positively homogeneous if and only if $\gph(T)$
is a cone. If $T_{1},T_{2}:X\rightrightarrows Y$ are two set-valued
maps such that $T_{1}(w)\subset T_{2}(w)$ for all $w\in X$, then
we write this property as $T_{1}\subset T_{2}$. We denote the set-valued
map $T(-\cdot):X\rightrightarrows Y$ to be $T(-\cdot)(w):=T(-w)$. 

We recall the definition of inner limits of a set valued map.\emph{}
\begin{defn}
When $S:X\rightrightarrows Y$ is a set-valued map, we say that\[
\limsup_{x\to\bar{x}}S(x):=\{y\in Y\mid\liminf_{x\xrightarrow{}\bar{x}}d(y,S(x))=0\}\]
is the \emph{outer limit} of $S$ at $\bar{x}$ and \[
\liminf_{x\rightarrow\bar{x}}S(x):=\{y\in Y\mid\lim_{x\xrightarrow{}\bar{x}}d(y,S(x))=0\}\]
is the \emph{inner limit }of $S$ at $\bar{x}$. When the outer and
inner limits coincide, it is called the \emph{limit}.
\end{defn}
The above definition of limits are equivalent to those in \cite{RW98}
by \cite[Exercise 4.2]{RW98}. We recall the definitions of outer
and inner semicontinuity.
\begin{defn}
\label{def:osc-isc-cty}For a closed-valued mapping $S:X\rightrightarrows Y$
and a point $\bar{x}\in X$:
\begin{enumerate}
\item $S$ is \emph{upper semicontinuous} \emph{at $\bar{x}$} if for any
open set $U$ such that $S(\bar{x})\subset U$, there is a neighborhood
$V$ of $\bar{x}$ such that $S(x)\subset U$ for all $x\in V$.
\item $S$ is\emph{ outer semicontinuous} \emph{at $\bar{x}$} if for any
open set $U$ such that $S(\bar{x})\subset U$ and $\rho>0$, there
is a neighborhood $V$ of $\bar{x}$ such that $S(x)\cap\rho\mathbb{B}\subset U$
for all $x\in V$. 
\item $S$ is \emph{inner semicontinuous at $\bar{x}$ }if $S(\bar{x})\subset\liminf_{x\rightarrow\bar{x}}S(x)$.
\item $S$ is \emph{continuous at $\bar{x}$ }if it is both outer and inner
semicontinuous there.
\end{enumerate}
\end{defn}
We caution that the terminology used to denote upper semicontinuity
and outer semicontinuity is not consistent in the literature. We choose
to define outer semicontinuity in this manner because this is the
property of outer semicontinuity that we will use in the proofs of
the chain rule in Theorem \ref{thm:Chain-rule} and in the proofs
in Section \ref{sec:Strict--T-from-T}. A set-valued map that is not
closed-valued can still satisfy the condition for outer semicontinuity
in (2). For example, consider $S:\mathbb{R}\rightrightarrows\mathbb{R}$
at $0$, where $S$ is defined by \[
S(x):=\begin{cases}
(-1,1) & \mbox{if }x=0\\
0 & \mbox{otherwise.}\end{cases}\]
In finite dimensions, outer semicontinuity is equivalent to the notation
in \cite{RW98} through \cite[Proposition 5.12]{RW98} and the result
below.
\begin{prop}
For a closed-valued mapping $S:X\rightrightarrows\mathbb{R}^{m}$
and a point $\bar{x}\in X$, $S$ is outer semicontinuous at $\bar{x}$
if and only if either of the following equivalent conditions hold:
\begin{enumerate}
\item [(2$^*$)]$\limsup_{x\to\bar{x}}S(x)\subset S(\bar{x})$.
\item [(2$^\prime$)] For any $\epsilon>0$ and $\rho>0$, there is a neighborhood
$V$ of $\bar{x}$ such that \[
S(x)\cap\rho\mathbb{B}\subset S(\bar{x})+\epsilon\mathbb{B}\mbox{ for all }x\in V.\]

\end{enumerate}
\end{prop}
\begin{proof}
The equivalence of $(2)$ (in Definition \ref{def:osc-isc-cty}) and
$(2^{*})$ mimics the proof of \cite[Proposition 5.12]{RW98}. (Note
that the proof may not be extended to the case where $S:X\rightrightarrows Y$
and $Y$ is infinite dimensional since it relies on the compactness
of the closed unit ball in $Y$.)

The implication $(2)\Rightarrow(2^{\prime})$ is straightforward,
so we prove the opposite direction. Suppose that an open set $U$
is such that $S(\bar{x})\subset U$ and $\rho$ is chosen arbitrarily.
For each $y\in S(\bar{x})\cap(\rho+1)\mathbb{B}$, the value $\epsilon_{y}:=\sup\{\epsilon\mid\mathbb{B}(y,\epsilon)\subset U\}$
is positive. We prove that $\bar{\epsilon}:=\inf\{\epsilon_{y}\mid y\in S(\bar{x})\cap(\rho+1)\mathbb{B})\}>0$.
Suppose otherwise. Then there is a sequence $\{y_{i}\}\subset S(\bar{x})\cap(\rho+1)\mathbb{B}$
such that $\epsilon_{y_{i}}\to0$. By the compactness of $S(\bar{x})\cap(\rho+1)\mathbb{B}$,
we can assume, by taking a subsequence if necessary, that $y_{i}\to\bar{y}\in S(\bar{x})\cap(\rho+1)\mathbb{B}$.
But $\epsilon_{\bar{y}}>0$ contradicts $\epsilon_{y_{i}}\to0$. This
implies that $[S(\bar{x})\cap(\rho+1)\mathbb{B}]+\bar{\epsilon}\mathbb{B}\subset U$.
We may reduce $\bar{\epsilon}$ so that $\bar{\epsilon}<1$. 

By assumption $(2^{\prime})$, for our choice of $\rho$ and $\bar{\epsilon}$,
we can find a neighborhood $V$ of $\bar{x}$ such that \begin{equation}
S(x)\cap\rho\mathbb{B}\subset S(\bar{x})+\bar{\epsilon}\mathbb{B}\mbox{ for all }x\in V.\label{eq:osc-cond}\end{equation}
Then \eqref{eq:osc-cond} implies $S(x)\cap\rho\mathbb{B}\subset[S(\bar{x})\cap(\rho+1)\mathbb{B}]+\bar{\epsilon}\mathbb{B}\subset U$
for all $x\in V$, which proves what we need. 
\end{proof}
Outer semicontinuity is better suited to handle set-valued maps with
unbounded value sets $S(x)$. For example, the set-valued map $S:\mathbb{R}\rightrightarrows\mathbb{R}^{2}$
defined by \[
S(\theta):=\{(t\cos\theta,t\sin\theta)\mid t\geq0\}\]
(see \cite[Figure 5-7]{RW98} or \cite[Page 27]{BI08}) is not upper
semicontinuous anywhere but is outer semicontinuous, and in fact continuous,
everywhere. When $S(\bar{x})$ is bounded, upper and outer semicontinuity
are equivalent. We will not use upper semicontinuity in this paper.

To simplify the notation, given any map $T:X\rightrightarrows Y$
and constant $\delta>0$, denote $(T+\delta):X\rightrightarrows Y$
to be the map \[
(T+\delta)(w):=T(w)+\delta|w|\mathbb{B}.\]

\section{\label{sec:single-valued-diff}Generalized differentiability of single-valued
maps}

The emphasis of this section is the generalized differentiability
of single-valued maps $f:X\rightarrow Y$. Much of the theory is already
in \cite{Iof81}, but we concentrate on the key results that we will
extend for the set-valued case in later sections. We now begin with
our first definition of generalized differentiability. 
\begin{defn}
($T$-differentiability) \label{def:single-valued-T-diff}Let $T:X\rightrightarrows Y$
be a positively homogeneous set-valued map. We say that $f:X\rightarrow Y$
is\emph{ $T$-differentiable} at $\bar{x}$ if for any $\delta>0$,
there exists a neighborhood $V$ of $\bar{x}$ such that \[
f(x)\in f(\bar{x})+T(x-\bar{x})+\delta|x-\bar{x}|\mathbb{B}\mbox{ for all }x\in V.\]
We say that $f:X\rightarrow Y$ is \emph{strictly $T$-differentiable}
at $\bar{x}$ if for any $\delta>0$, there exists a neighborhood
$V$ of $\bar{x}$ such that \[
f(x)\in f(x^{\prime})+T(x-x^{\prime})+\delta|x-x^{\prime}|\mathbb{B}\mbox{ for all }x,x^{\prime}\in V.\]

\end{defn}
The map $T:X\rightrightarrows Y$ is referred to as a \emph{prederivative}
when $f$ is $T$-differentiable, and as the \emph{strict prederivative}
when $f$ is strictly $T$-differentiable in \cite[Definition 9.1]{Iof81}.
The definition of $T$-differentiability includes the familiar concepts
of differentiability and Lipschitz continuity as special cases.
\begin{example}
(Examples of $T$-differentiability) Let $f:X\rightarrow Y$ be a
single-valued map.
\begin{enumerate}
\item When $T:X\rightarrow Y$ is a (single-valued) continuous linear map,
$T$-differentiability is precisely Fr\'{e}chet differentiability
with derivative $T$, and strict $T$-dif\-fer\-en\-ti\-a\-bil\-i\-ty
is precisely strict differentiability with derivative $T$.
\item When $T:X\rightrightarrows Y$ is defined by $T(w)=\kappa|w|\mathbb{B}$,
$T$-differentiability is precisely calmness with modulus $\kappa$
(i.e., $|f(x)-f(\bar{x})|\leq\kappa|x-\bar{x}|+o(|x-\bar{x}|)$),
and strict $T$-differentiability is precisely local Lipschitz continuity
with modulus $\kappa$.
\end{enumerate}
\end{example}
Strict $T$-differentiability is more robust than $T$-differentiability.
The function $f:\mathbb{R}\rightarrow\mathbb{R}$ defined by \[
f(x):=\begin{cases}
x^{2}\sin(x^{-2}) & \mbox{if }x\neq0\\
0 & \mbox{if }x=0\end{cases}\]
 is Fr\'{e}chet differentiable at $0$ but not Lipschitz there. 

With the right choice of $T$, $T$-differentiability can also handle
inequalities.
\begin{example}
(Inequalities with $T$-differentiability) Let $f:X\rightarrow\mathbb{R}$
be a single-valued map.\end{example}
\begin{enumerate}
\item $f:X\rightarrow\mathbb{R}$ is \emph{calm from below} at $\bar{x}$
with modulus $\kappa$, i.e., $f(x)\geq f(\bar{x})-\kappa|x-\bar{x}|+o(|x-\bar{x}|)$,
if and only if $f$ is $T$-differentiable there, where $T:X\rightrightarrows\mathbb{R}$
is defined by $T(w)=[-\kappa|w|,\infty)$.
\item The vector $v\in X^{*}$ is a \emph{Fr\'{e}chet subgradient} of $f$
at $\bar{x}$ if and only if $f$ is $T$-differentiable there, where
$T:X\rightrightarrows\mathbb{R}$ is defined by $T(w)=[\left\langle v,w\right\rangle ,\infty)$.
\item The convex set $C\subset X^{*}$ is a subset of the \emph{Fr\'{e}chet
subdifferential} of $f$ at $\bar{x}$ if and only if $f$ is $T$-differentiable
there, where $T:X\rightrightarrows\mathbb{R}$ is defined by $T(w)=[\sup_{v\in C}\left\langle v,w\right\rangle ,\infty)$.
\end{enumerate}
It is clear from the definitions that if $f$ is $T_{1}$-differentiable
at $\bar{x}$, and $T_{2}$ satisfies $T_{2}\supset T_{1}$, then
$f$ is $T_{2}$-differentiable at $\bar{x}$ as well. 

At this point, we mention connections to other notions of generalized
derivatives for single-valued functions close to the definition of
(strict) $T$-dif\-fer\-en\-ti\-a\-bil\-i\-ty. Semidifferentiability,
as is recorded in \cite[Definition 7.20]{RW98}, can be traced back
to \cite{Penot79}, and is equivalent to the case where $T:X\rightrightarrows Y$
is continuous and single-valued (see \cite[Section 9D]{RW98}). Semidifferentiability
for single-valued maps is not to be confused with semidifferentiability
for set-valued maps defined in Definition \ref{def:set-valued-semidiff}. 

We now look at a slightly nontrivial example involving the Clarke
subdifferential \cite{Cla75}. 
\begin{defn}
\cite[Section 2.1]{Cla83} \label{def:Clarke-subdiff}(Clarke subdifferential)
Let $X$ be a Banach space. Suppose $f:X\rightarrow\mathbb{R}$ is
locally Lipschitz around $\bar{x}$. The \emph{Clarke generalized
directional derivative} of $f$ at $\bar{x}$ in the direction $v\in X$
is defined by\[
f^{\circ}(\bar{x};v)=\limsup_{t\searrow0,x\rightarrow\bar{x}}\frac{f(x+tv)-f(x)}{t},\]
where $x\in X$ and $t$ is a positive scalar. The \emph{Clarke subdifferential}
(or \emph{generalized subdifferential}) of $f$ at $\bar{x}$, denoted
by $\partial_{C}f(\bar{x})$, is the convex subset of the dual space
$X^{*}$ given by\[
\{\zeta\in X^{*}\mid f^{\circ}(\bar{x};v)\geq\left\langle \zeta,v\right\rangle \mbox{ for all }v\in X\}.\]

\end{defn}
The Clarke subdifferential enjoys a mean value theorem. For $C\subset X^{*}$,
define the set $\left\langle C,w\right\rangle $ by $\left\{ \left\langle c,w\right\rangle \mid c\in C\right\} $.
The following result is due to Lebourg \cite{Leb75}, which has since
been generalized in other ways. We will recall other subdifferentials
in Definition \ref{def:more-subdif}.
\begin{thm}
\label{thm:Lebourg-MVT}\cite{Leb75} (Nonsmooth mean value theorem)
Suppose $x_{1},x_{2}\in X$ and $f$ is Lipschitz on an open set containing
the line segment $[x_{1},x_{2}]$. Then there exists a point $u\in(x_{1},x_{2})$
such that \[
f(x_{2})-f(x_{1})\in\left\langle \partial_{C}f(u),x_{2}-x_{1}\right\rangle .\]

\end{thm}
We show how the Clarke subdifferential relates to $T$-dif\-fer\-en\-ti\-a\-bil\-i\-ty.
For extensions and a more general treatment of the following result,
we refer the reader to \cite[Sections 9, 10]{Iof81}.
\begin{thm}
(Clarke subdifferential and $T$-differentiability) Let $f:X\rightarrow\mathbb{R}$
be a Lipschitz function, and $C$ be a nonempty weak{*}-compact convex
subset of $X^{*}$. If $f$ is strictly $T$-differentiable at $\bar{x}$,
where $T:X\rightrightarrows\mathbb{R}$ is defined by $T(w)=\left\langle C,w\right\rangle $,
then $\partial_{C}f(\bar{x})\subset C$. The converse holds if the
map $x\mapsto\partial_{C}f(x)$ is outer semicontinuous at $\bar{x}$,
which is the case when $X=\mathbb{R}^{n}$.\end{thm}
\begin{proof}
If $f$ is strictly $T$-differentiable at $\bar{x}$, then for any
direction $h\in X$, the Clarke directional derivative is\begin{eqnarray*}
\max_{v\in\partial_{C}f(\bar{x})}\left\langle v,h\right\rangle  & = & f^{\circ}(\bar{x};h)\\
 & = & \limsup_{x\rightarrow\bar{x},t\searrow0}\frac{f(x+th)-f(x)}{t}\\
 & \leq & \max_{v\in C}\left\langle v,h\right\rangle .\end{eqnarray*}
Suppose on the contrary that $v^{\prime}\in\partial_{C}f(\bar{x})\backslash C$.
Then there exists a direction $h^{\prime}$ and some $\alpha\in\mathbb{R}$
such that $\left\langle v^{\prime},h^{\prime}\right\rangle >\alpha$
but $\max_{v\in C}\left\langle v,h^{\prime}\right\rangle \leq\alpha$.
This is a contradiction, which shows that $\partial_{C}f(\bar{x})\subset C$. 

We now prove the converse. It is well-known that when $X=\mathbb{R}^{n}$,
the Clarke subdifferential mapping of a Lipschitz function is outer
semicontinuous. See \cite{RW98} for example. Suppose that $\partial_{C}f(\bar{x})\subset C$.
By outer semicontinuity, given any $\delta>0$, there is some $\epsilon>0$
such that $\partial_{C}f(x)\subset C+\delta\mathbb{B}$ for all $x\in\mathbb{B}(\bar{x},\epsilon)$.
Theorem \ref{thm:Lebourg-MVT} states that for any points $x_{1},x_{2}\in\mathbb{B}(\bar{x},\epsilon)$,
there is an $x^{\prime}\in(x_{1},x_{2})$ such that \[
f(x_{1})-f(x_{2})\subset\left\langle \partial_{C}f(x^{\prime}),x_{1}-x_{2}\right\rangle .\]
This immediately implies that $f(x_{1})\in f(x_{2})+\left\langle C,x_{1}-x_{2}\right\rangle +\delta|x_{1}-x_{2}|\mathbb{B}$,
and hence strict $T$-differentiability. 
\end{proof}
It is well known that in finite dimensions, the Clarke subdifferential
is the convex hull of the limit of gradients taken over where the
function is differentiable. We now recall Rademacher's theorem.
\begin{thm}
(Rademacher's Theorem) Let $O\subset\mathbb{R}^{n}$ be open, and
let $f:O\rightarrow\mathbb{R}^{m}$ be Lipschitz. Let $D$ be the
subset of $O$ consisting of the points where $F$ is differentiable.
Then $O\backslash D$ is a set of measure zero in $\mathbb{R}^{n}$.
In particular, $D$ is dense in $O$, i.e., $\cl\, D\supset O$. 
\end{thm}
Closely related to the Clarke subdifferential is the Clarke Jacobian
that was first introduced in \cite{Cla76}. 
\begin{defn}
(Clarke Jacobian) Let $f:O\rightarrow\mathbb{R}^{m}$ be Lipschitz,
with $O\subset\mathbb{R}^{n}$ open, and let $D\subset O$ consist
of the points where $f$ is differentiable. The \emph{Clarke Jacobian}
(or \emph{generalized Jacobian}) at $\bar{x}$ is defined by\[
\bar{\nabla}f(\bar{x}):=\conv\{A\in\mathbb{R}^{m\times n}:\exists x_{i}\rightarrow\bar{x}\mbox{ with }x_{i}\in D,\nabla f(x_{i})\rightarrow A\}.\]

\end{defn}
It is clear from Rademacher's Theorem that the Clarke Jacobian is
a nonempty, compact set of matrices. For $m=1$, it is well-known
that the Clarke Jacobian reduces to the Clarke subdifferential. The
following result is equivalent to \cite[Proposition 10.9]{Iof81},
and is a generalization of a result that is well known for $m=1$.
\begin{thm}
\label{thm:Jacobian-T-diff}(Clarke Jacobian and $T$-differentiability)
Let $f:O\rightarrow\mathbb{R}^{m}$ be Lipschitz on an open set $O\subset\mathbb{R}^{n}$.
At each $\bar{x}\in O$, $f$ is $T$-differentiable at $\bar{x}$,
where $T:\mathbb{R}^{n}\rightrightarrows\mathbb{R}^{m}$ is defined
by\[
T(w):=\{Aw\mid A\in\bar{\nabla}f(\bar{x})\}.\]
\end{thm}
\begin{proof}
In the case $O=\mathbb{R}^{n}$, a useful equivalent condition for
a set $C\subset\mathbb{R}^{n}$ to be of measure zero is this: with
respect to any vector $w\neq\mathbf{0}$, $C$ is of measure zero
if and only if the set $\{\tau\mid x+\tau w\in C\}\subset\mathbb{R}$
is of measure zero (in the one-dimensional sense) for all $x$ outside
a set of measure zero. Let $D$ be the subset of $O$ on which $f$
is differentiable. 

Fix $w\neq\mathbf{0}$ and let $x\in O$ and $\tau>0$ be such that
$[x,x+\tau w]\subset O$, and $\{t\in[0,\tau]\mid[x,x+\tau w]\cap(O\backslash D)\}$
has (one dimensional) measure zero. Then the function $\varphi(t)=f(x+tw)$
is Lipschitz continuous for $t\in[0,\tau]$. Lipschitz continuity
guarantees that $\varphi(\tau)=\varphi(0)+\int_{0}^{\tau}\varphi^{\prime}(t)dt$,
and in this integral a negligible set of $t$ values can be disregarded.
Thus, the integral is unaffected if we concentrate on $t$ values
such that $x+tw\in D$, in which case $\varphi^{\prime}(t)=\nabla f(x+tw)(w)$.

For any $\epsilon>0$, we can consider a neighborhood $O^{\prime}\subset O$
of $\bar{x}$ so that whenever $x\in O^{\prime}\cap D$, then $\nabla f(x)\in\bar{\nabla}f(\bar{x})+\epsilon\mathbb{B}$.
We have\begin{eqnarray*}
f(x+\tau w) & = & \varphi(\tau)\\
 & = & \varphi(0)+\int_{0}^{\tau}\varphi^{\prime}(t)dt\\
 & = & f(x)+\int_{0}^{\tau}\nabla f(x+tw)(w)dt\\
 & \subset & f(x)+\int_{0}^{\tau}T(w)+\epsilon|w|\mathbb{B}dt\\
 & \subset & f(x)+T(\tau w)+\epsilon|\tau w|\mathbb{B}.\end{eqnarray*}
The case where $\nabla f(x+tw)$ does not exist for all $t$ in a
set of nonzero measure can be treated easily by perturbing $x$. This
establishes the $T$-differentiability of $f$.
\end{proof}
In Theorem \ref{thm:Jacobian-T-diff}, it is clear that there can
be no closed convex valued positively homogeneous map $T^{\prime}\subsetneq T$
such that $f$ is $T^{\prime}$-differentiable at $\bar{x}$.

\section{\label{sec:set-valued-diff}Generalized differentiability of set-valued
maps}

In this section, we move on to define the generalized differentiability
of set-valued maps and state some basic properties. Here is the first
definition of the differentiability of a set-valued map.
\begin{defn}
($T$-differentiability) \label{def:set-valued-diff}Let $T:X\rightrightarrows Y$
be a positively homogeneous set-valued map.

(a) We say that $S:X\rightrightarrows Y$ is \emph{outer $T$-differentiable}
at $\bar{x}$ if for any $\delta>0$, there exists a neighborhood
$V$ of $\bar{x}$ such that \[
S(x)\subset S(\bar{x})+T(x-\bar{x})+\delta|x-\bar{x}|\mathbb{B}\mbox{ for all }x\in V.\]
It is\emph{ inner $T$-differentiable }at $\bar{x}$ if the formula
above is replaced by \[
S(\bar{x})\subset S(x)-T(x-\bar{x})+\delta|x-\bar{x}|\mathbb{B}\mbox{ for all }x\in V.\]
It is\emph{ $T$-differentiable} at $\bar{x}$ if it is both outer
$T$-differentiable and inner $T$-differentiable. 

(b) We say that $S:X\rightrightarrows Y$ is \emph{strictly} \emph{$T$-differentiable}
at $\bar{x}$ if for any $\delta>0$, there exists a neighborhood
$V$ of $\bar{x}$ such that \[
S(x)\subset S(x^{\prime})+T(x-x^{\prime})+\delta|x-x^{\prime}|\mathbb{B}\mbox{ for all }x,x^{\prime}\in V.\]

\end{defn}
It is elementary that if $S:X\rightrightarrows Y$ is a single-valued
map $S:X\rightarrow Y$, then the definitions of outer $T$-differentiability,
inner $T$-differentiability and $T$-differentiability in Definition
\ref{def:set-valued-diff}(a) coincide. 

The case $T(w):=\kappa|w|\mathbb{B}$, where $\kappa\geq0$ is finite,
is well studied in variational analysis. In the definitions of calmness
and Lipschitz continuity below, we recall the notation for $\kappa$
commonly used in variational analysis. Calmness was first referred
to as {}``upper Lipschitzian'' by Robinson \cite{Rob81}, who established
this property for polyhedral mappings. 
\begin{defn}
(Calmness and Lipschitzness) \label{def:calm-Lipschitz}(a) We say
that $S:X\rightrightarrows Y$ is \emph{calm} at $\bar{x}$ if there
exists a neighborhood $V$ of $\bar{x}$ and $\kappa\geq0$ such that
\[
S(x)\subset S(\bar{x})+\kappa|x-\bar{x}|\mathbb{B}\mbox{ for all }x\in V.\]
The infimum of all constants $\kappa$ is the \emph{calmness modulus},
denoted by $\calm\, S(\bar{x})$. Calmness at $\bar{x}$ is equivalent
to outer $T$-differentiability at $\bar{x}$, where $T:X\rightrightarrows Y$
is defined by $T(w):=[\calm\, S(\bar{x})]|w|\mathbb{B}$.

(b) We say that $S:X\rightrightarrows Y$ is \emph{Lipschitz} at $\bar{x}$
if there exists a neighborhood $V$ of $\bar{x}$ and $\kappa>0$
such that \[
S(x)\subset S(x^{\prime})+\kappa|x-x^{\prime}|\mathbb{B}\mbox{ for all }x,x^{\prime}\in V.\]
The infimum of all constants $\kappa$ is the \emph{Lipschitz modulus},
denoted by $\lip\, S(\bar{x})$. Lipschitz continuity at $\bar{x}$
is equivalent to strict $T$-differentiability at $\bar{x}$, where
$T:X\rightrightarrows Y$ is defined by $T(w):=[\lip\, S(\bar{x})]|w|\mathbb{B}$.
\end{defn}
Consider the case where there is some $\kappa\geq0$ such that $T:X\rightrightarrows Y$
satisfies $T(w)\subset\kappa|w|\mathbb{B}$ for all $w\in X$. It
follows straight from the definitions that if $S:X\rightrightarrows Y$
is outer $T$-differentiable and maps to compact sets, then it is
outer semicontinuous. The same relations hold for inner $T$-differentiability
and inner semicontinuity, and for $T$-differentiability and continuity.
We remark that strict $T$-differentiability implies strict continuity
in the sense of \cite[Definition 9.28]{RW98}, but not vice versa.
Their difference is analogous to the difference between upper semicontinuity
and outer semicontinuity. 

We motivate this definition of set-valued differentiability with Proposition
\ref{pro:T-diff-example}, whose proof is straightforward. We now
recall the Pompieu-Hausdorff distance.
\begin{defn}
(Pompieu Hausdorff distance) For $C,D\subset X$ closed and nonempty,
the \emph{Pompieu-Hausdorff distance} between $C$ and $D$ is the
quantity\[
\mathbf{d}(C,D):=\inf\{\eta\mid C\subset D+\eta\mathbb{B},D\subset C+\eta\mathbb{B}\}.\]

\end{defn}
The Pompieu-Hausdorff distance is a metric on compact subsets of $X$.
In fact, the motivation of calmness and Lipschitz continuity in Definition
\ref{def:calm-Lipschitz} comes from the Pompieu-Hausdorff distance.
To motivate the definition of $T$-differentiability, we note the
following result.
\begin{prop}
\label{pro:T-diff-example}(Single-valued $T$-differentiability)
Suppose $S:X\rightrightarrows Y$ is closed-valued. 
\begin{enumerate}
\item Let $T:X\rightarrow Y$ be a single-valued map. Then $S$ is $T$-differentiable
at $\bar{x}$ if and only if for any $\delta>0$, there is a neighborhood
$V$ of $\bar{x}$ such that \[
\mathbf{d}(S(\bar{x})+T(x-\bar{x}),S(x))\leq\delta|x-\bar{x}|\mbox{ for all }x\in V.\]

\item Let $T:X\rightarrow Y$ be a single-valued map such that $T=-T(-\cdot)$.
Then $S$ is strictly $T$-differentiable at $\bar{x}$ if and only
if for any $\delta>0$, there is a neighborhood $V$ of $\bar{x}$
such that \[
\mathbf{d}(S(x^{\prime})+T(x-x^{\prime}),S(x))\leq\delta|x-x^{\prime}|\mbox{ for all }x,x^{\prime}\in V.\]

\end{enumerate}
\end{prop}
We now make a remark on the Pompieu-Hausdorff distance that is in
the spirit of the main idea in this paper.
\begin{rem}
(More precise measurement of sets) We can rewrite the Pompieu-Hausdorff
distance as\[
\mathbf{d}(C,D)=\inf\{\eta\mid C\subset D+E_{1},D\subset C+E_{2},E_{1}\subset\eta\mathbb{B},E_{2}\subset\eta\mathbb{B}\}.\]
In certain situations, it might be useful to study sets $E_{1}$ and
$E_{2}$ for which $C\subset D+E_{1}$ and $D\subset C+E_{2}$ instead
of just taking them to be $\eta\mathbb{B}$.
\end{rem}
As is well-known in set-valued analysis, setting restrictions on the
range gives a sharper analysis at the points of interest. We make
the following definitions with this in mind.
\begin{defn}
(Pseudo $T$-differentiability) \label{def:pseudo-T-diff}Let $S:X\rightrightarrows Y$
be a set-valued map such that $\bar{y}\in S(\bar{x})$.
\begin{enumerate}
\item Let $T:X\rightrightarrows Y$ be a set-valued map. We say that $S$
is \emph{pseudo outer $T$-differentiable} at $\bar{x}$ for $\bar{y}$
if for any $\delta>0$, there exist neighborhoods $V$ of $\bar{x}$
and $W$ of $\bar{y}$ such that \[
S(x)\cap W\subset S(\bar{x})+T(x-\bar{x})+\delta|x-\bar{x}|\mathbb{B}\mbox{ for all }x\in V.\]
 It is \emph{pseudo inner $T$-differentiable} at $\bar{x}$ for $\bar{y}$
if \[
S(\bar{x})\cap W\subset S(x)-T(x-\bar{x})+\delta|x-\bar{x}|\mathbb{B}\mbox{ for all }x\in V.\]
It is \emph{pseudo $T$-differentiable} at $\bar{x}$ for $\bar{y}$
if it is both pseudo outer  $T$-dif\-fer\-en\-tia\-ble and pseudo
inner $T$-differentiable there. 
\item Let $T:X\rightrightarrows Y$ be a positively homogeneous set-valued
map. We say that $S$ is \emph{pseudo strictly} \emph{$T$-differentiable}
at $\bar{x}$ for $\bar{y}$ if for any $\delta>0$, there exist neighborhoods
$V$ of $\bar{x}$ and $W$ of $\bar{y}$ such that \[
S(x)\cap W\subset S(x^{\prime})+T(x-x^{\prime})+\delta|x-x^{\prime}|\mathbb{B}\mbox{ for all }x,x^{\prime}\in V.\]

\end{enumerate}
\end{defn}
Again, the case $T(w):=\kappa|w|\mathbb{B}$ is of particular interest.
The Aubin property was first introduced as pseudo-Lipschitzness in
\cite{Aub84}.
\begin{defn}
(Calmness and Aubin property) \label{def:calm-aubin}Let $S:X\rightrightarrows Y$
be a set-valued map such that $\bar{y}\in S(\bar{x})$.
\begin{enumerate}
\item We say that $S$ is \emph{calm} at $\bar{x}$ for $\bar{y}$ if there
exist neighborhoods $V$ of $\bar{x}$ and $W$ of $\bar{y}$, and
$\kappa\geq0$ such that \[
S(x)\cap W\subset S(\bar{x})+\kappa|x-\bar{x}|\mathbb{B}\mbox{ for all }x\in V.\]
The infimum of all such constants $\kappa$ is the \emph{calmness
modulus}, denoted by $\calm\, S(\bar{x}\mid\bar{y})$. Calmness is
precisely pseudo outer $T$-differentiability, where $T:X\rightrightarrows Y$
is defined by $T(w):=[\calm\, S(\bar{x}\mid\bar{y})]|w|\mathbb{B}$.
\item We say that $S:X\rightrightarrows Y$ has the \emph{Aubin Property}
at $\bar{x}$ for $\bar{y}$ if there exist neighborhoods $V$ of
$\bar{x}$ and $W$ of $\bar{y}$, and $\kappa\geq0$ such that \[
S(x)\cap W\subset S(x^{\prime})+\kappa|x-x^{\prime}|\mathbb{B}\mbox{ for all }x,x^{\prime}\in V.\]
The infimum of all such constants $\kappa$ is the \emph{graphical
modulus}, denoted by $\lip\, S(\bar{x}\mid\bar{y})$. The Aubin property
is also known as the \emph{pseudo-Lipschitz property }and as the \emph{Lipschitz-like
property}. The Aubin property is precisely pseudo strict $T$-differentiability,
where $T:X\rightrightarrows Y$ is defined by $T(w):=[\lip\, S(\bar{x}\mid\bar{y})]|w|\mathbb{B}$.
\end{enumerate}
\end{defn}

While the Clarke subdifferential is unique, the generalized derivative
$T$ need not be.
\begin{example}
\label{exa:Nonunique-linear}(Nonuniqueness and failure of intersections)
Let $S:\mathbb{R}^{2}\rightrightarrows\mathbb{R}^{2}$ be defined
by $S(x,y)=\{x\}\times\mathbb{R}$. Let $T_{1},T_{2}:\mathbb{R}^{2}\rightarrow\mathbb{R}^{2}$
be defined by $T_{1}(x,y)=(x,y)$ and $T_{2}(x,y)=(x,0)$. It is clear
that $S$ is (pseudo) strictly $T_{1}$-differentiable and (pseudo)
strictly $T_{2}$-differentiable at all points, but it is not (pseudo)
$(T_{1}\cap T_{2})$-differentiable anywhere.
\end{example}
The positively homogeneous map $T:\mathbb{R}^{n}\rightrightarrows\mathbb{R}^{m}$
defined from the Clarke Jacobian in Theorem \ref{thm:Jacobian-T-diff}
is defined as the union of linear functions, and thus satisfies $T=-T(-\cdot)$.
In such a case $T$-differentiability implies $-T(-\cdot)$-differentiability
for single valued maps. But this need not be the case for set-valued
maps.
\begin{example}
\label{exa:double-neg-failure}(Failure of set-valued $-T(-\cdot)$-differentiability)
Let $S:\mathbb{R}\rightrightarrows\mathbb{R}$ be defined by $S(x):=(-\infty,x]$,
and let $T:\mathbb{R}\rightrightarrows\mathbb{R}$ be defined by \[
T(w)=\begin{cases}
\{0\} & \mbox{ if }w\leq0\\
{}[-w,w] & \mbox{ if }w\geq0.\end{cases}\]
Here, $S$ is strictly $T$-differentiable at $x$ for all $x\in\mathbb{R}$,
and it is pseudo strictly $T$-differentiable at $x$ for $x$ for
all $x\in\mathbb{R}$. However, $S$ is neither outer (or inner) $-T(-\cdot)$-differentiable
anywhere, nor pseudo outer (or inner) $-T(-\cdot)$-differentiable
at $x$ for $x$ for any $x\in\mathbb{R}$. 
\end{example}
We remark that while inner $T$-differentiability is defined so that
Proposition \ref{pro:T-diff-example} holds, this definition of inner
$T$-differentiability does not satisfy the property that strict $T$-differentiability
implies inner $T$-differentiability in general. It actually implies
inner $-T(-\cdot)$-differentiability. The same holds for pseudo strict
$T$-differentiability and pseudo inner $T$-differentiability, or
more correctly, pseudo inner $-T(-\cdot)$-differentiability. An example
where this occurs is the function $S:\mathbb{R}\rightrightarrows\mathbb{R}$
as defined in Example \ref{exa:double-neg-failure}. Fortunately,
inner $T$-differentiability does not play a huge role in this paper.

Much of the current methods for set-valued differentiation are motivated
by looking at the tangent cones of the graph of the set-valued map.
See the discussion in \cite[Chapter 5]{AF90} on the different forms
of set-valued differentiation obtained by taking different kinds of
tangent cones of the graph. The notions of semidifferentiability \cite{Penot84}
and proto-differentiability \cite{Roc89} are based on this idea.
We recall the definitions of semidifferentiability and proto-differentiability
from \cite[Page 331- 332]{RW98}. See also the techniques in \cite[Chapter 5]{AF90}.
We now point out the lack of relation between pseudo $T$-differentiability
and these methods by observing the finite dimensional case. 
\begin{defn}
\label{def:set-valued-semidiff}\cite{Penot84} (Semidifferentiability)
For a set-valued map $S:\mathbb{R}^{n}\rightrightarrows\mathbb{R}^{m}$,
let $\bar{x}\in\dom(S)$, $\bar{y}\in S(\bar{x})$ and $\bar{w}\in\mathbb{R}^{n}$.
If the limit \[
\lim_{\tau\searrow0,w\rightarrow\bar{w}}\frac{S(\bar{x}+\tau w)-\bar{y}}{\tau}\]
exists, then we say that it is the \emph{semiderivative }at $\bar{x}$
for $\bar{y}$ and $\bar{w}$. If the semiderivative exists for every
vector $\bar{w}\in\mathbb{R}^{n}$, then $S$ is \emph{semidifferentiable
}at $\bar{x}$ for $\bar{y}$ with derivative $DS(\bar{x}\mid\bar{y}):\mathbb{R}^{n}\rightrightarrows\mathbb{R}^{m}$
defined by $DS(\bar{x}\mid\bar{y})(\bar{w})$ being equal to the limit
defined above. 
\begin{defn}
\cite{Roc89} (Proto-differentiability) A mapping $S:\mathbb{R}^{n}\rightrightarrows\mathbb{R}^{m}$
is said to be \emph{proto-differentiable} at $\bar{x}$ for an element
$\bar{y}\in S(\bar{x})$ if \[
DS(\bar{x}\mid\bar{y})(\bar{w})=\limsup_{\tau\searrow0,w\rightarrow\bar{w}}\frac{S(\bar{x}+\tau w)-\bar{y}}{\tau},\]
and there exist for each $\bar{z}\in DS(\bar{x}\mid\bar{u})(\bar{w})$
and choice of $\tau_{i}\searrow0$ sequences \[
w_{i}\to\bar{w}\mbox{ and }z_{i}\to\bar{z}\mbox{ with }z_{i}\in[S(\bar{x}+\tau_{i}w_{i})-\bar{y}]/\tau_{i}.\]

\end{defn}
\end{defn}
For this paper, it is sufficient to note that semidifferentiability
implies proto-differentiability through \cite[Exercise 8.43 (a)(d)]{RW98}.

The definition of semidifferentiability for set-valued maps is not
to be confused with the semidifferentiability defined for single-valued
maps before Definition \ref{def:Clarke-subdiff}. Clearly, the limits
in the definitions above are cones, and the derivative $DS(\bar{x}\mid\bar{y})$
are both positively homogeneous maps. We present an example where
$S:\mathbb{R}\rightrightarrows\mathbb{R}^{2}$ is not pseudo outer
$DS(\bar{x}\mid\bar{y})$-differentiable at $\bar{x}$ for $\bar{y}$.
\begin{prop}
\label{pro:semi_T} (No $T$-differentiability from semidifferentiability)
Consider the set-valued map $S:\mathbb{R}\rightrightarrows\mathbb{R}^{2}$
defined by \begin{eqnarray*}
S(x) & = & \begin{cases}
\{(t,\sqrt{-t})\mid{\color{red}x}\leq t\leq0\}\cup\big(\{x\}\times[\sqrt{-x},\infty)\big) & \mbox{ if }x\leq0\\
\{(t,\sqrt{t})\mid0\leq t\leq x\}\cup\big(\{x\}\times[\sqrt{x},\infty)\big) & \mbox{ if }x\geq0.\end{cases}\end{eqnarray*}
This map is semidifferentiable (and hence proto-differentiable) at
$0$ for $\mathbf{0}$, but not pseudo outer $DS(0\mid\mathbf{0})$-differentiable
at $0$ for $\mathbf{0}$.\end{prop}
\begin{proof}
The map $S$ is semidifferentiable (and hence proto-differentiable)
at $0$ for $\mathbf{0}$, with semiderivative $DS(0\mid\mathbf{0}):\mathbb{R}\rightrightarrows\mathbb{R}^{2}$
defined by $DS(0\mid\mathbf{0})(w)=\{0\}\times[0,\infty)$ for all
$w\in\mathbb{R}$. 

Suppose on the contrary that $S$ is pseudo outer $DS(0\mid\mathbf{0})$-differentiable
at $0$ for $\mathbf{0}$. Then for any $\delta>0$, there are neighborhoods
$V$ of $0$ and $W$ of $\mathbf{0}$ such that\begin{equation}
S(x)\cap W\subset S(0)+[DS(0\mid\mathbf{0})+\delta](x)\mbox{ for all }x\in V.\label{eq:pseu-o-no}\end{equation}
The neighborhood $W$ contains some point of the form $(0,\alpha)$
in its interior. We can find a neighborhood $W_{\alpha}$ of $(0,\alpha)$
such that $W_{\alpha}\subset W$ and \[
S(x)\cap W_{\alpha}=[\{x\}\times(-\infty,\infty)]\cap W_{\alpha},\]
as illustrated in Figure \ref{fig:semi_T}. An easy calculation shows
that \[
S(0)+[DS(0\mid\mathbf{0})+\delta](x)=[-\delta x,\delta x]\times[0,\infty).\]
This means that if $\delta<1$, then $S(x)\cap W_{\alpha}\not\subset S(0)+[DS(0\mid\mathbf{0})+\delta](x)$
for all $x$ sufficiently close to $0$, hence violating \eqref{eq:pseu-o-no}.
Therefore $S$ is not pseudo outer $DS(0\mid\mathbf{0})$-differentiable
at $0$ for $\mathbf{0}$.
\end{proof}
\begin{figure}
\includegraphics[scale=0.5]{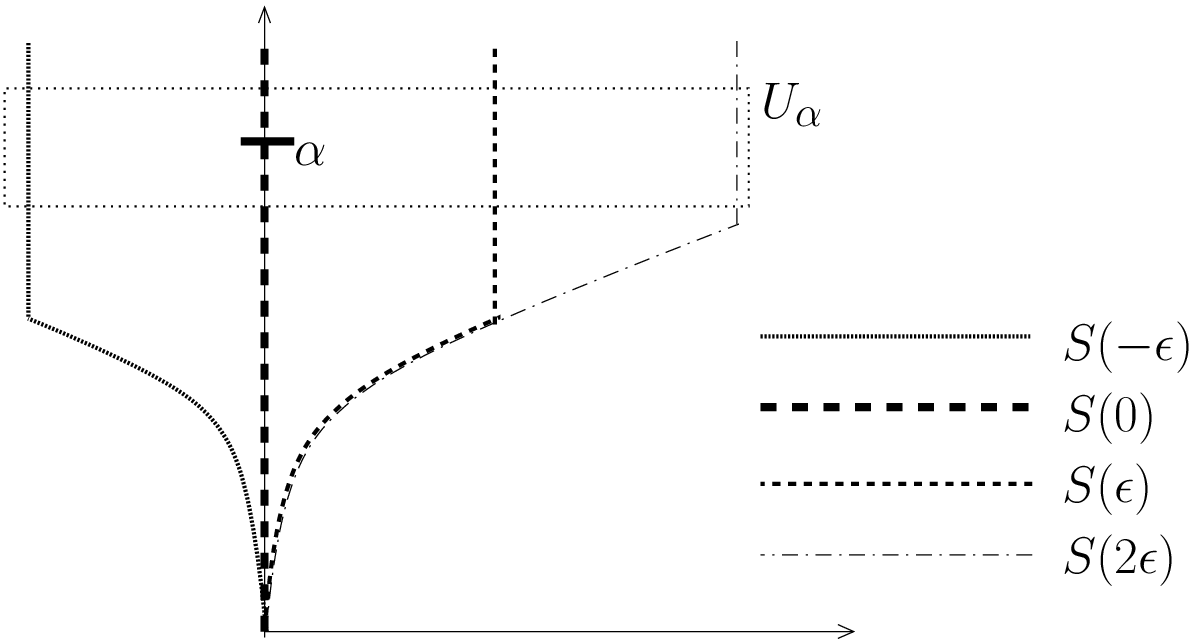}

\caption{\label{fig:semi_T}In Proposition \ref{pro:semi_T}, for any $\alpha>0$,
we can find a neighborhood $W_{\alpha}$ of the point $(0,\alpha)$
such that the map $S:\mathbb{R}\rightrightarrows\mathbb{R}^{2}$ satisfies
$S(x)\cap W_{\alpha}=[\{x\}\times(-\infty,\infty)]\cap W_{\alpha}$. }

\end{figure}

One may expect that if $S$ is $T$-differentiable at $\bar{x}$ for
$\bar{y}$, then $DS(\bar{x}\mid\bar{y})\subset T$. The following
example shows that this is not the case. 
\begin{example}
(No semidifferentiability from $T$-differentiability) Consider $S:\mathbb{R}\rightrightarrows\mathbb{R}$
defined by $S(x)=\mathbb{R}$. Clearly $S$ is pseudo $T$-differentiable
at $x$ for any $y\in\mathbb{R}$, where $T:\mathbb{R}\rightarrow\mathbb{R}$
is defined by $T(w)=0$ for all $w\in\mathbb{R}$. But the semiderivative
$DS(\bar{x}\mid\bar{y}):\mathbb{R}\rightrightarrows\mathbb{R}$, or
equivalently the proto-derivative, is equal to $S$, which gives $DS(\bar{x}\mid\bar{y})\not\subset T$
for all $\bar{x},\bar{y}\in\mathbb{R}$.
\end{example}

\section{\label{sec:Calculus}Calculus of $T$-differentiability}

In this section, we prove a chain rule and a sum rule similar to the
coderivative calculus in \cite[Section 10H]{RW98}, which was first
studied in \cite{Mor94}. Let us first introduce the outer norm of
a positively homogeneous set-valued map.
\begin{defn}
(Outer norm) If $T:X\rightrightarrows Y$ is a positively homogeneous
map, then the \emph{outer norm of $T$} is\[
|T|^{+}:=\sup_{|x|\leq1}\sup_{y\in T(x)}|y|=\sup\{\left|y\right|\mid y\in T(x),\left|x\right|\leq1\}.\]
Clearly $|T|^{+}<\infty$ implies $T(\mathbf{0})=\{\mathbf{0}\}$.
When $T(\mathbf{0})=\{\mathbf{0}\}$, $|T|^{+}=\calm\, T(\mathbf{0}\mid\mathbf{0})=\calm\, T(\mathbf{0})$.
\end{defn}
Here is a chain rule for $T$-differentiable functions. In the statements
of the results in this section, the positively homogeneous map $T_{\bar{x}\to y}$
refers to a generalized derivative that arises from considering the
point $y\in F(\bar{x})$. Other such maps $T$ are similarly defined.
\begin{thm}
\label{thm:Chain-rule}(Chain rule) Let $F:X\rightrightarrows Y$,
$G:Y\rightrightarrows Z$, and $\bar{z}\in G\circ F(\bar{x})$. Suppose
the following conditions hold
\begin{enumerate}
\item $F$ is pseudo outer  $T_{\bar{x}\to y}$-differentiable at $\bar{x}$
for $y$ for all $y$ in $G^{-1}(\bar{z})\cap F(\bar{x})$.
\item $G$ is pseudo strictly $T_{y\to\bar{z}}$-differentiable at $y$
for $\bar{z}$ for all $y$ in $G^{-1}(\bar{z})\cap F(\bar{x})$. 
\item The set $G^{-1}(\bar{z})\cap F(\bar{x})\subset Y$ is compact.
\item The map $(x,z)\mapsto G^{-1}(z)\cap F(x)$ is outer semicontinuous
at $(\bar{x},\bar{z})$.
\item $\alpha:=\sup_{y\in G^{-1}(\bar{z})\cap F(\bar{x})}\left|T_{\bar{x}\to y}\right|^{+}$
is finite.
\item For all $y\in G^{-1}(\bar{z})\cap F(\bar{x})$, $T_{y\rightarrow\bar{z}}(\mathbf{0})=\{\mathbf{0}\}$,
and $\beta$ is finite, where \[
\beta:=\sup_{y\in G^{-1}(\bar{z})\cap F(\bar{x})}\lip\, T_{y\to\bar{z}}(\mathbf{0}).\]

\end{enumerate}
Then $G\circ F$ is pseudo outer $T$-differentiable at $\bar{x}$
for $\bar{z}$, where $T:X\rightrightarrows Z$ is defined by \begin{equation}
T:=\bigcup_{y\in G^{-1}(\bar{z})\cap F(\bar{x})}T_{y\to\bar{z}}\circ T_{\bar{x}\to y}.\label{eq:chain-rule-formula}\end{equation}
The function $G\circ F$ is pseudo strictly $T$-differentiable for
$T:X\rightrightarrows Y$ defined in \eqref{eq:chain-rule-formula}
if in statement \emph{(1)}, $F$ were pseudo strictly $T_{\bar{x}\to y}$-differentiable
at $\bar{x}$ for $y$ instead. \end{thm}
\begin{proof}
We shall prove only the result for $F$ being pseudo outer  $T$-differentiable.
The proof for pseudo strict $T$-differentiability is almost exactly
the same. Choose any $\delta>0$. Since (2) holds, for each $y\in G^{-1}(\bar{z})\cap F(\bar{x})$,
we can find some open convex neighborhoods $V_{y}^{\prime}$ of $y$
and $W$ of $\bar{z}$ such that \[
G(y^{\prime})\cap W\subset G(y^{\prime\prime})+T_{y\to\bar{z}}(y^{\prime}-y^{\prime\prime})+\delta|y^{\prime}-y^{\prime\prime}|\mathbb{B}\mbox{ for all }y^{\prime},y^{\prime\prime}\in V_{y}^{\prime}.\]
Next, since (1) holds, for each $y\in G^{-1}(\bar{z})\cap F(\bar{x})$,
we can find open convex neighborhoods $U$ of $\bar{x}$ and $V_{y}\subset V_{y}^{\prime}$
of $y$ such that \[
F(x)\cap V_{y}\subset[(F(\bar{x})\cap V_{y}^{\prime})+T_{\bar{x}\to y}(x-\bar{x})+\delta|x-\bar{x}|\mathbb{B}]\cap V_{y}\mbox{ for all }x\in U.\]
Since $G^{-1}(\bar{z})\cap F(\bar{x})$ is compact by (3), there are
finitely many $y_{i}\in G^{-1}(\bar{z})\cap F(\bar{x})$ such that
$G^{-1}(\bar{z})\cap F(\bar{x})\subset\bigcup_{i}V_{y_{i}}$. By taking
finitely many intersections if necessary, the neighborhoods $U$ and
$W$ can be assumed to be independent of $y_{i}$.

Let $V:=\bigcup_{i}V_{y_{i}}$. Since the map $(x,z)\mapsto G^{-1}(z)\cap F(x)$
is outer semicontinuous at $(\bar{x},\bar{z})$ by (4), we can reduce
$U$ and $W$ if necessary so that $G^{-1}(W)\cap F(U)\subset V.$
This implies that for any $x\in U$, \begin{eqnarray*}
G(F(x))\cap W & = & G(F(x)\cap V)\cap W\\
 & = & \left(\bigcup_{i}G\big(F(x)\cap V_{y_{i}}\big)\right)\cap W.\end{eqnarray*}
We have\begin{eqnarray}
 &  & G(F(x)\cap V_{y_{i}})\cap W\nonumber \\
 & \subset & G\big([(F(\bar{x})\cap V_{y_{i}}^{\prime})+T_{\bar{x}\to y_{i}}(x-\bar{x})+\delta|x-\bar{x}|\mathbb{B}]\cap V_{y_{i}}\big)\cap W\nonumber \\
 & \subset & G\big(F(\bar{x})\cap V_{y_{i}}^{\prime}\big)+T_{y_{i}\to\bar{z}}\big(T_{\bar{x}\to y_{i}}(x-\bar{x})+\delta|x-\bar{x}|\mathbb{B}\big)\label{eq:chain-rule-finale}\\
 &  & \qquad\qquad+\delta\big|T_{\bar{x}\to y_{i}}(x-\bar{x})+\delta|x-\bar{x}|\mathbb{B}\big|\mathbb{B}\nonumber \\
 & \subset & G\circ F(\bar{x})+T_{y_{i}\to\bar{z}}\circ T_{\bar{x}\to y_{i}}(x-\bar{x})+\delta\beta|x-\bar{x}|\mathbb{B}+\delta(\alpha+\delta)|x-\bar{x}|\mathbb{B}\nonumber \\
 & = & G\circ F(\bar{x})+T_{y_{i}\to\bar{z}}\circ T_{\bar{x}\to y_{i}}(x-\bar{x})+\delta(\alpha+\beta+\delta)|x-\bar{x}|\mathbb{B}.\nonumber \end{eqnarray}
We may assume that $\delta$ is small enough so that $\alpha+\beta+\delta$
is bounded from above by some finite constant $\theta$. Choosing
$y_{i}$ over all $i$ gives \begin{eqnarray*}
 &  & \bigcup_{i}G\big(F(x)\cap V_{y_{i}}\big)\cap W\\
 & \subset & G\circ F(\bar{x})+\left(\bigcup_{i}T_{y_{i}\to\bar{z}}\circ T_{\bar{x}\to y_{i}}(x-\bar{x})\right)+\delta\theta|x-\bar{x}|\mathbb{B}\\
 & \subset & G\circ F(\bar{x})+\left(\bigcup_{y\in G^{-1}(\bar{z})\cap F(\bar{x})}T_{y\to\bar{z}}\circ T_{\bar{x}\to y}(x-\bar{x})\right)+\delta\theta|x-\bar{x}|\mathbb{B}.\end{eqnarray*}
This completes the proof of the theorem.
\end{proof}
In Theorem \ref{thm:Chain-rule}, we require $G:Y\rightrightarrows Z$
to be pseudo strictly $T$-differentiable for the appropriate $T:Y\rightrightarrows Z$.
If $G:Y\rightrightarrows Z$ were pseudo outer $T$-differentiable,
then we can still obtain a result for the case where $F:X\rightrightarrows Y$
is single-valued.
\begin{prop}
(Chain rule) Let $f:X\to Y$ and $G:Y\rightrightarrows Z$, $\bar{y}=f(\bar{x})$,
and $\bar{z}\in G(\bar{y})$. Suppose the following conditions hold
\begin{enumerate}
\item $f$ is  $T_{\bar{x}\to\bar{y}}$-differentiable at $\bar{x}$ for
$\bar{y}$.
\item $G$ is pseudo outer $T_{\bar{y}\to\bar{z}}$-differentiable at $\bar{y}$
for $\bar{z}$.
\item $\left|T_{\bar{x}\to\bar{y}}\right|^{+}$ is finite.
\item $T_{\bar{y}\rightarrow\bar{z}}(\mathbf{0})=\{\mathbf{0}\}$, and $\lip\, T_{\bar{y}\to\bar{z}}(\mathbf{0})$
is finite.
\end{enumerate}
Then $G\circ f$ is pseudo outer $T$-differentiable at $\bar{x}$
for $\bar{z}$, where $T:X\rightrightarrows Z$ is defined by $T=T_{\bar{y}\to\bar{z}}\circ T_{\bar{x}\to\bar{y}}.$\end{prop}
\begin{proof}
Choose some $\delta>0$. By condition  (1), we can find a neighborhood
$U$ of $\bar{x}$ such that\[
f(x)\in f(\bar{x})+T_{\bar{x}\to\bar{y}}(x-\bar{x})+\delta|x-\bar{x}|\mathbb{B}\mbox{ for all }x\in U.\]
Let $V$ be such that $f(\bar{x})+T_{\bar{x}\to\bar{y}}(x-\bar{x})+\delta|x-\bar{x}|\mathbb{B}\subset V$
for all $x\in U$. By (2), we can shrink $U$ and $V$ if necessary
so that there is a neighborhood $W$ of $\bar{z}$ such that \[
G(y)\cap W\subset G(\bar{y})+T_{\bar{y}\to\bar{z}}(y-\bar{y})+\delta|y-\bar{y}|\mathbb{B}\mbox{ for all }y\in V.\]
A calculation similar to \eqref{eq:chain-rule-finale} concludes the
proof. 
\end{proof}
From the chain rule, we can infer a sum rule.
\begin{cor}
(Sum rule) Let $S_{i}:X\rightrightarrows Y$ for $i=1,\dots,p$, and
$\bar{y}\in\sum_{i=1}^{p}S_{i}(\bar{x})$. Define $F:X\rightrightarrows Y^{p}$
by $F(x)=(S_{1}(x),\dots,S_{p}(x))$, and define $g:Y^{p}\rightarrow Y$
to be the linear map mapping to the sum of the $p$ elements in $Y^{p}$.
Suppose the following conditions hold.
\begin{enumerate}
\item $S_{i}$ is pseudo outer $T_{\bar{x}\to y_{i}}^{i}$-differentiable
at $\bar{x}$ for $y_{i}$ whenever $y_{i}\in S_{i}(\bar{x})$ and
$y_{1}+\cdots+y_{p}=\bar{y}$.
\item The set $\{(y_{1},\dots,y_{p})\mid y_{i}\in S_{i}(\bar{x}),\, y_{1}+\cdots+y_{p}=\bar{y}\}\subset Y^{p}$
is compact.
\item The map $(x,y)\mapsto g^{-1}(y)\cap F(x)$ is outer semicontinuous
at $(\bar{x},\bar{y})$.
\item $\alpha_{i}:=\sup_{y_{i}\in\Pi_{i}(g^{-1}(\bar{z})\cap F(\bar{x}))}\left|T_{\bar{x}\to y_{i}}^{i}\right|^{+}$
is finite for each $i=1,\dots,p$. Here, $\Pi_{i}:Y^{p}\rightarrow Y$
is the projection onto the $i$th coordinate.
\end{enumerate}
Then $g\circ F:X\rightrightarrows Y$, which is the sum of the maps
$S_{i}$, is $T$-differentiable at $\bar{x}$ for $\bar{y}$, where
$T:X\rightrightarrows Y$ is defined by\begin{equation}
T:=\bigcup_{\begin{array}{c}
y_{1}+\cdots+y_{p}=\bar{y}\\
y_{i}\in S_{i}(\bar{x})\end{array}}\left(\sum_{i}T_{\bar{x}\rightarrow y_{i}}^{i}\right).\label{eq:sum-rule-formula}\end{equation}
The function $g\circ F$ is pseudo strictly $T$-differentiable for
$T:X\rightrightarrows Y$ defined in \eqref{eq:sum-rule-formula}
if in statement \emph{(1)}, $S_{i}$ were pseudo strictly $T_{\bar{x}\to y_{i}}^{i}$-differentiable
at $\bar{x}$ for $y_{i}$ instead. \end{cor}
\begin{proof}
For $y_{i}\in S_{i}(\bar{x})$ for all $i$, define $T_{(y_{1},\dots,y_{p})}:X\rightrightarrows Y^{p}$
by \[
T_{(y_{1},\dots,y_{p})}(w):=\big(T_{\bar{x}\to y_{1}}^{1}(w),\dots,T_{\bar{x}\to y_{p}}^{p}(w)\big).\]
Condition (1) implies that the map $F$ is pseudo outer $T_{(y_{1},\dots,y_{p})}$-differentiable
at $\bar{x}$ for $(y_{1},\dots,y_{p})$. We now proceed to apply
the chain rule in Theorem \ref{thm:Chain-rule}. Since $g$ is a linear
function, the conditions for $g$ needed for the chain rule are satisfied.
The rest of the conditions in this result are just the appropriate
conditions in the chain rule rephrased. The case of pseudo strict
$T$-differentiability is similar.
\end{proof}
Note that we have focused on pseudo (outer/ strict) $T$-differentiability
so far in this section. The relation between pseudo (strict/ outer/
inner) $T$-differentiability and (strict/ outer/ inner) $T$-differentiability
is illustrated by the following theorem. We say that $S:X\rightrightarrows Y$
is \emph{locally compact} around $\bar{x}\in\mbox{dom}(F)$ if there
is a neighborhood $O$ of $\bar{x}$ and a compact set $C\subset Y$
such that $S(O)\subset C$. 
\begin{thm}
\label{thm:T-diff-from-pseudo}($T$-differentiability from pseudo
$T$-differentiability) Let $S:D\rightrightarrows Y$ be a closed-valued
outer semicontinuous map on a closed domain $D\subset X$. Suppose
$S$ is locally compact around $\bar{x}\in D$. Then $S$ is outer
 $T$-differentiable at $\bar{x}$ if and only if $S$ is pseudo outer
 $T$-differentiable at $\bar{x}$ for all $\bar{y}\in S(\bar{x})$. 

An analogous statement holds for (strict/ inner) $T$-differentiability
and pseudo (strict/ inner) $T$-differentiability.\end{thm}
\begin{rem}
In fact, the hypothesis of outer semicontinuity in Theorem \ref{thm:T-diff-from-pseudo}
can be weakened: We can assume that $S$ is \emph{closed} at $\bar{x}$:
For every $y\notin S(\bar{x})$, there are neighborhoods $U$ of $\bar{x}$
and $V$ of $y$ such that $S(x)\cap V=\emptyset$ for all $x\in U$.
If $S$ is outer semicontinuous, then $\gph(S)$ is closed by \cite[Proposition 1.4.8]{AF90},
which implies that $S$ is closed at $\bar{x}$. The proof of Theorem
\ref{thm:T-diff-from-pseudo} can be easily adapted from the proof
of \cite[Theorem 1.42]{Mor06}, which traces its roots in the finite
dimensional case to \cite{Roc85}.
\end{rem}
Other calculus rules that are important are the Cartesian product
of set-valued maps, and rules for unions. These two operations are
simple to formulate and prove. For intersections of set-valued maps,
we feel it is more effective to look at the normal cones of the intersections
of the graph of the appropriate functions and apply the Mordukhovich
criterion. See Section \ref{sec:Boris-extended}. We close this section
by referring the reader to \cite{Roc85} for more applications of
set-valued chain rules.

\section{\label{sec:Boris-extended}The Mordukhovich criterion}

As we have seen in Section \ref{sec:set-valued-diff}, the Aubin property
gives a sharper analysis for the Lipschitz continuity of set-valued
maps. An effective tool for calculating the graphical modulus (for
the Aubin property) is the coderivative defined in Definition \ref{def:coderivative},
which is a generalization of the adjoint linear operator of linear
functions. Coderivatives enjoy an effective calculus, and can be easily
calculated for set-valued maps whose graphs are defined by smooth
maps. The relationship between the Aubin property and the coderivatives
is referred to as the Mordukhovich criterion in \cite{RW98}. For
a history of the Mordukhovich criterion, see the bibliography in \cite{RW98},
which in turn cited \cite{DMO81,Iof87,Kru88,Mor88,War81}, and also
Commentaries 1.4.6--1.4.9, 4.5.2 and 4.5.6 of \cite{Mor06}. The aim
of this section is to show that coderivatives in fact give directional
behavior that is captured in the language of pseudo strict $T$-differentiability.

We now recall the classical definition of the normal cones and coderivatives
in finite dimensions, which were first introduced in \cite{Mor76}
and \cite{Mor80} respectively.
\begin{defn}
\cite{Mor76} (Normals) Let $C\subset\mathbb{R}^{n}$ and $\bar{x}\in C$.
A vector $v$ is \emph{normal to $C$ at $\bar{x}$ in the regular
sense}, or a \emph{regular normal}, written $v\in\hat{N}_{C}(\bar{x})$,
if \[
\left\langle v,x-\bar{x}\right\rangle \leq o(|x-\bar{x}|)\mbox{ for }x\in C.\]
It is \emph{normal to $C$ at $\bar{x}$ in the general sense} (or
\emph{limiting normal}, or \emph{Mordukhovich normal}, or simply a
\emph{normal }vector), written $v\in N_{C}(\bar{x})$, if there are
sequences $x_{i}\xrightarrow[C]{}\bar{x}$ and $v_{i}\rightarrow v$
with $v_{i}\in\hat{N}_{C}(x_{i})$. 
\begin{defn}
\cite{Mor80} (Coderivatives) \label{def:coderivative}Consider a
mapping $S:\mathbb{R}^{n}\rightrightarrows\mathbb{R}^{m}$, a point
$\bar{x}\in\dom(S)$ and $\bar{y}\in S(\bar{x})$. The \emph{coderivative
(or limiting coderivative, or Mordukhovich coderivative) at $\bar{x}$
for $\bar{y}$} is the mapping $D^{*}S(\bar{x}\mid\bar{y}):\mathbb{R}^{m}\rightrightarrows\mathbb{R}^{n}$
defined by\[
v\in D^{*}S(\bar{x}\mid\bar{y})(z)\iff(v,-z)\in N_{\scriptsize\gph(S)}(\bar{x},\bar{y}).\]

\end{defn}
\end{defn}
In the case where $S$ is a smooth map, the coderivative $D^{*}S(\bar{x}\mid S(\bar{x}))$
is the adjoint of the derivative mapping there. 

Next, we recall the definition of the regular subdifferential and
general subdifferential, which are important in the proof in the main
result of this section.
\begin{defn}
\label{def:more-subdif}(Subdifferentials) Consider a function $f:\mathbb{R}^{n}\rightarrow\mathbb{R}$
and a point $\bar{x}\in\mathbb{R}^{n}$. For a vector $v\in\mathbb{R}^{n}$, 
\begin{enumerate}
\item [(a)]$v$ is a \emph{regular subgradient} of $f$ at $\bar{x}$,
written $v\in\hat{\partial}f(\bar{x})$, if \[
f(x)\geq f(\bar{x})+\left\langle v,x-\bar{x}\right\rangle +o(|x-\bar{x}|);\]

\item [(b)]$v$ is a \emph{(general) subgradient }of $f$ at $\bar{x}$,
written $v\in\partial f(\bar{x})$, if there are sequences $x_{i}\rightarrow\bar{x}$
and $v_{i}\in\hat{\partial}f(x_{i})$ with $v_{i}\rightarrow v$ and
$f(x_{i})\rightarrow f(\bar{x})$.
\item [(c)]The sets $\hat{\partial}f(\bar{x})$ and $\partial f(\bar{x})$
are the \emph{regular subdifferential} (or \emph{Fr\'{e}chet subgradient})
and the \emph{(general) subdifferential} (or \emph{limiting subdifferential},
or \emph{Mordukhovich subdifferential}) at $\bar{x}$ respectively.
\end{enumerate}
\end{defn}
We now present our main result of this section.
\begin{thm}
\label{thm:Boris-criterion-precise}(Mordukhovich criterion revisited)
Consider $S:\mathbb{R}^{n}\rightrightarrows\mathbb{R}^{m}$, $\bar{x}\in\dom(S)$
and $\bar{y}\in S(\bar{x})$. Suppose $\gph(S)$ is locally closed
at $(\bar{x},\bar{y})$. If $D^{*}S(\bar{x}\mid\bar{y})(\mathbf{0})=\{\mathbf{0}\}$,
or equivalently, $|D^{*}S(\bar{x}\mid\bar{y})|^{+}<\infty$, then
$S$ is pseudo strictly $T$-differentiable at $\bar{x}$ for $\bar{y}$,
where $T:\mathbb{R}^{n}\rightrightarrows\mathbb{R}^{m}$ is defined
by\[
T(w)=\kappa(w)\mathbb{B}^{m},\]
where $\kappa:\mathbb{R}^{n}\rightarrow\mathbb{R}_{+}$ is defined
by \begin{eqnarray*}
\kappa(w) & := & \max\left\{ \left\langle -v,w\right\rangle \mid v\in D^{*}S(\bar{x}\mid\bar{y})(\mathbb{B}^{m})\right\} \\
 & = & \max\left\{ \left\langle -v,w\right\rangle \mid(v,-z)\in N_{\scriptsize\gph(S)}(\bar{x},\bar{y})\mbox{ for some }|z|\leq1\right\} .\end{eqnarray*}

\end{thm}
The assumptions of Theorem \ref{thm:Boris-criterion-precise} are
the same as that of the Mordukhovich Criterion as stated in \cite[Theorem 9.40]{RW98}.
The Mordukhovich criterion establishes the equivalence between the
Aubin property (with graphical modulus $\kappa^{*}:=\max_{|w|\leq1}\kappa(w)$)
and the outer norm of the coderivative. Theorem \ref{thm:Boris-criterion-precise}
shows that the worst case Lipschitzian behavior need not occur in
all directions. 

Rockafellar \cite{Roc85} established the relationship between the
Aubin property of $S$ at $\bar{x}$ for $\bar{y}$ and the Lipschitz
continuity properties of $d_{\tilde{y}}(\cdot)$. (See \cite[Exercise 9.37]{RW98}.)
The key to the proof of Theorem \ref{thm:Boris-criterion-precise}
is that the nonsmooth mean value theorem on $d_{\tilde{y}}(\cdot)$
gives us more information on the continuity properties of $S$. 
\begin{proof}
(of Theorem \ref{thm:Boris-criterion-precise}) We highlight only
the parts we need to add to the proof of sufficiency as given in \cite[Theorem 9.40]{RW98},
though with some changes to the variable names. Let $\bar{\kappa}=\left|D^{*}S(\bar{x}\mid\bar{y})\right|^{+}$.
By the definition of $\bar{\kappa}$ and the coderivatives, for any
$\theta>0$, there exists $\delta_{0}>0$ and $\epsilon_{0}>0$ for
which\[
N_{\scriptsize\gph(S)}(\hat{x},\hat{y})\cap(\mathbb{R}^{n}\times\mathbb{B}^{m})\subset(\bar{\kappa}+\theta)\mathbb{B}^{n}\times\mathbb{B}^{m}\mbox{ for all }\hat{x}\in\mathbb{B}(\bar{x},\delta_{0}),\hat{y}\in\mathbb{B}(\bar{y},\epsilon_{0}).\]
(The above statement is just \cite[9(22)]{RW98} rephrased.) The mapping
to normal cones is outer semicontinuous, so \begin{equation}
\begin{array}{l}
N_{\scriptsize\gph(S)}(\hat{x},\hat{y})\cap[(\bar{\kappa}+\theta)\mathbb{B}^{n}\times\mathbb{B}^{m}]\subset N_{\scriptsize\gph(S)}(\bar{x},\bar{y})+\theta\mathbb{B}^{n+m}\\
\qquad\qquad\qquad\qquad\qquad\qquad\mbox{ for all }\hat{x}\in\mathbb{B}(\bar{x},\delta_{0}),\hat{y}\in\mathbb{B}(\bar{y},\epsilon_{0}).\end{array}\label{eq:osc-normals}\end{equation}
Define the map $d_{\tilde{y}}(x):\mathbb{R}^{n}\rightarrow\mathbb{R}$
by $d_{\tilde{y}}(x):=d(\tilde{y},S(x))$. In their proof, the condition
in \cite[9(23)]{RW98} can be written more precisely as:\begin{equation}
\left\{ \begin{array}{l}
\mbox{if }v\in\hat{\partial}d_{\tilde{y}}(\hat{x})\mbox{ with }|\hat{x}-\bar{x}|\leq\delta_{0}\\
\mbox{and }d_{\tilde{y}}\left(\hat{x}\right)+|\tilde{y}-\bar{y}|\leq\epsilon_{0}\mbox{, then }\\
\mbox{there exists }z\in\mathbb{B}\mbox{ with }(v,-z)\in N_{\scriptsize\gph(S)}(\hat{x},\hat{y}),\end{array}\right.\label{eq:9(23)-better}\end{equation}
where $\hat{y}$ is any point of $S(\hat{x})$ nearest to $\tilde{y}$,
that is $|\hat{y}-\tilde{y}|=d_{\tilde{y}}(\hat{x})$. Moving on,
they proved that the map $(x,y)\mapsto d(y,S(x))$ is Lipschitz continuous
near $(\bar{x},\bar{y})$, or more precisely, there exists some $\lambda>0$
and $\mu>0$ such that \begin{equation}
d(y,S(x))\leq\lambda(|x-\bar{x}|+|y-\bar{y}|)\mbox{ when }|x-\bar{x}|\leq\mu\mbox{ and }|y-\bar{y}|\leq\mu.\label{eq:9.40-formula}\end{equation}

Furnished with this, we can choose $\delta>0$ and $\epsilon>0$ small
enough that $2\delta\leq\min\{\mu,\delta_{0}\}$, $\epsilon\leq\min\{\mu,\epsilon_{0}/2\}$
and $\lambda(2\delta+\epsilon)\leq\epsilon_{0}/2$. If $v\in\hat{\partial}d_{\tilde{y}}(\hat{x})$,
$|\hat{x}-\bar{x}|\leq2\delta$ and $|\tilde{y}-\bar{y}|\leq\epsilon$,
then by \eqref{eq:9.40-formula}, we have\[
d_{\tilde{y}}(\hat{x})\leq\lambda(|\hat{x}-\bar{x}|+|\tilde{y}-\bar{y}|)\leq\lambda(2\delta+\epsilon)\leq\frac{\epsilon_{0}}{2}.\]
Formula \eqref{eq:9(23)-better} then tells us that for any $v\in\hat{\partial}d_{\tilde{y}}(\hat{x})$,
there exists some $z\in\mathbb{B}$ such that $(v,-z)\in N_{\scriptsize\gph(S)}(\hat{x},\hat{y})$. 

By \eqref{eq:osc-normals}, there is some $(\tilde{v},-\tilde{z})\in N_{\scriptsize\gph(S)}(\bar{x},\bar{y})$
such that $\left|(\tilde{v},-\tilde{z})-(v,-z)\right|\leq\theta$.
This gives $\tilde{v}\in D^{*}S(\bar{x}\mid\bar{y})((1+\theta)\mathbb{B}^{m})$,
or $\frac{1}{1+\theta}\tilde{v}\in D^{*}S(\bar{x}\mid\bar{y})(\mathbb{B}^{m})$.
Further arithmetic gives $|v-\frac{1}{1+\theta}\tilde{v}|\leq|v-\tilde{v}|+|\frac{\theta}{1+\theta}\tilde{v}|\leq\theta(1+\frac{\bar{\kappa}}{1+\theta})$,
so $v\in D^{*}S(\bar{x}\mid\bar{y})(\mathbb{B}^{m})+\theta(1+\frac{\bar{\kappa}}{1+\theta})\mathbb{B}^{n}$.
Let $\bar{\theta}$ be $\theta(1+\frac{\bar{\kappa}}{1+\theta})$.
We thus have $\partial d_{\tilde{y}}(\hat{x})\subset D^{*}S(\bar{x}\mid\bar{y})(\mathbb{B})+\bar{\theta}\mathbb{B}$.
The function $d_{\tilde{y}}(\cdot)$ is Lipschitz, so the Clarke subdifferential
equals $\conv(\partial d_{\tilde{y}}(\cdot))$ in $\mathbb{B}(\bar{y},\epsilon)$.

Choose any two points $x,x^{\prime}\in\mathbb{B}(\bar{x},\delta)$.
By the nonsmooth mean value theorem (Theorem \ref{thm:Lebourg-MVT}),
we have $d_{\tilde{y}}(x)-d_{\tilde{y}}(x^{\prime})=\left\langle v,x-x^{\prime}\right\rangle $
for some $v\in\partial_{C}d_{\tilde{y}}(x_{\tau})=\conv(\partial d_{\tilde{y}}(x_{\tau}))$,
where $x_{\tau}=\tau x+(1-\tau)x^{\prime}$ for some $\tau\in(0,1)$.
Now, \begin{eqnarray*}
\left\langle v,x-x^{\prime}\right\rangle  & \leq & \max\left\{ \left\langle \tilde{v},x-x^{\prime}\right\rangle \mid\tilde{v}\in\partial_{C}d_{\tilde{y}}(x_{\tau})\right\} \\
 & \leq & \max\left\{ \left\langle \tilde{v},x-x^{\prime}\right\rangle \mid\tilde{v}\in D^{*}S(\bar{x}\mid\bar{u})(\mathbb{B})+\bar{\theta}\mathbb{B}\right\} \\
 & = & \kappa(x^{\prime}-x)+\bar{\theta}|x^{\prime}-x|.\end{eqnarray*}
This can be rephrased as $d_{\tilde{y}}(x^{\prime})\geq d_{\tilde{y}}(x)-\kappa(x^{\prime}-x)-\bar{\theta}|x^{\prime}-x|$.
As $\tilde{y}$ varies over all points in $\mathbb{B}(\bar{y},\epsilon)$,
this readily gives $S(x^{\prime})\cap\mathbb{B}(\bar{y},\epsilon)\subset S(x)+T(x^{\prime}-x)+\bar{\theta}|x^{\prime}-x|\mathbb{B}^{m}$,
which is what we seek to prove.
\end{proof}
We close this section with a remark on Theorem \ref{thm:Boris-criterion-precise}.
\begin{rem}
\label{rem:more-precise-T}(More precise $T$-differentiability) Suppose
$S:\mathbb{R}^{n}\rightrightarrows\mathbb{R}^{m}$, and $\bar{y}\in S(\bar{x})$.
To obtain a better positively homogeneous map $T:\mathbb{R}^{n}\rightrightarrows\mathbb{R}^{m}$
than the one stated in Theorem \ref{thm:Boris-criterion-precise},
one applies Theorem \ref{thm:Boris-criterion-precise} on the map
$g\circ S+f:\mathbb{R}^{n}\rightrightarrows\mathbb{R}^{m}$, where
$f:\mathbb{R}^{n}\rightarrow\mathbb{R}^{m}$ and $g:\mathbb{R}^{m}\rightarrow\mathbb{R}^{m}$
are linear maps and $g$ is invertible. This approach is equivalent
to looking at the normal cones of $\gph(g\circ S+f)$, which can also
be obtained by performing the linear map $(x,y)\mapsto(x,g(y)+f(x))$
on $\gph(S)\subset\mathbb{R}^{n}\times\mathbb{R}^{m}$. By appealing
to the calculus rules in Section \ref{sec:Calculus}, this gives another
$\tilde{T}:\mathbb{R}^{n}\rightrightarrows\mathbb{R}^{m}$ for which
$S$ is pseudo strictly $\tilde{T}$-differentiable. If we choose
finitely many $\{(f_{i},g_{i})\}_{i}$, then we can define $T^{\prime}:\mathbb{R}^{n}\rightrightarrows\mathbb{R}^{m}$
by\[
T^{\prime}(w)=\tilde{T}_{i}(w),\mbox{ where }i\mbox{ is determined uniquely by }w.\]
It is easy to see that $S$ is $T^{\prime}$-differentiable at $\bar{x}$
for $\bar{y}$.
\end{rem}
For a semialgebraic set-valued map $S:X\rightrightarrows\mathbb{R}^{m}$,
where $X\subset\mathbb{R}^{n}$ (i.e., a set-valued map whose graph
in $\mathbb{R}^{n}\times\mathbb{R}^{m}$ is a finite union of sets
defined by finitely many polynomial inequalities), the approach in
Remark \ref{rem:more-precise-T} was used in \cite{DP10} to prove
that for all $x\in X$ outside a set of smaller dimension than $X$,
given any $y\in S(x)$, we can find a linear map $T:\mathbb{R}^{n}\to\mathbb{R}^{m}$
such that $S$ is pseudo strictly $T$-differentiable at $x$ for
$y$. Note however, as illustrated in Example \ref{exa:Nonunique-linear},
that the linear map may not be unique.

\section{\label{sec:extended-metric-regularity}Metric regularity and open
covering}

For $S:X\rightrightarrows Y$, it is well known that the Aubin property
of $S^{-1}$ is related to the metric regularity and open covering
properties of $S$. In this section, we study metric regularity and
open covering in a more axiomatic manner with the help of $T$-differentiability,
proving new relations in these subjects. We caution that for much
of this section, we need $T:Y\rightrightarrows X$ instead of $T:X\rightrightarrows Y$.

We begin with our definitions of generalized metric regularity and
open covering.
\begin{defn}
($T$-metric regularity) \label{def:T-metric-regularity}Let $X$
and $Y$ be Banach spaces, $S:X\rightrightarrows Y$ be a set-valued
map  and $T:Y\rightrightarrows X$ be positively homogeneous. We
say that $S$ is \emph{$T$-metrically regular} at $(\bar{x},\bar{y})\in\gph(S)$
if for any $\delta>0$, there exist neighborhoods $V$ of $\bar{x}$
and $W$ of $\bar{y}$ and $r>0$ such that, for any $x\in V$ and
any set $A\subset r\mathbb{B}$, (or equivalently, for any set $A\subset r\mathbb{B}$
containing exactly one element), \begin{equation}
y\in[S(x)+A]\cap W\mbox{ implies }x\in S^{-1}(y)+(T+\delta)(A).\label{eq:T-MR-defn}\end{equation}

\begin{defn}
($T$-open covering)\label{def:T-openness} Let $X$ and $Y$ be Banach
spaces, $S:X\rightrightarrows Y$ be a set-valued map and $T:Y\rightrightarrows X$
be positively homogeneous. We say that $S$ is a \emph{$T$-open covering
}at $(\bar{x},\bar{y})\in\gph(S)$ if for any $\delta>0$, there exist
neighborhoods $V$ of $\bar{x}$ and $W$ of $\bar{y}$ and $r>0$
such that, for any $x\in V$ and any set $A\subset r\mathbb{B}$ (or
equivalently, for any set $A\subset r\mathbb{B}$ containing exactly
one element),\begin{equation}
[S(x)+A]\cap W\subset S\big(x+(T+\delta)(A)\big).\label{eq:T-open-defn}\end{equation}

\end{defn}
\end{defn}
Open covering is sometimes known as \emph{linear openness}. Setting
$T(w):=\kappa|w|\mathbb{B}$ in both cases reduce the above definitions
to the classical definitions of metric regularity and open covering
with modulus $\kappa$. An easy way to see the equivalence of $T$-metric
regularity of $S$ at $(\bar{x},\bar{y})\in\gph(S)$ with $T(w):=\kappa|w|\mathbb{B}$
and metric regularity of $S$ at $(\bar{x},\bar{y})$ with modulus
$\kappa$ is to observe that they are both equivalent to condition
(IT) in Theorem \ref{thm:T-equivalences} below. (This argument makes
use of the classical fact that $S^{-1}$ has the Aubin property at
$\bar{y}$ for $\bar{x}$ if and only if $S$ is metric regular at
$(\bar{x},\bar{y})$ with the same moduli.) The argument for open
coverings is similar. We will discuss metric regularity in greater
detail at the end of this section.
\begin{thm}
\label{thm:T-equivalences}($T$-metric regularity and $T$-openness)
Let $X$ and $Y$ be Banach spaces and $T:Y\rightrightarrows X$ be
a positively homogeneous map. For $S:X\rightrightarrows Y$, the following
are equivalent:
\begin{enumerate}
\item [(MR)]$S$ is $T(-\cdot)$-metric regular at $(\bar{x},\bar{y})\in\gph(S)$.
\item [(OC)]$S$ is a $(-T(-\cdot))$-open covering at $(\bar{x},\bar{y})\in\gph(S)$.
\item [(IT)]$S^{-1}$ is pseudo strictly $T$-differentiable at $\bar{y}$
for $\bar{x}$. 
\end{enumerate}
\end{thm}
\begin{proof}
 We first note that $x\in S^{-1}(y^{\prime})\cap V$ can be rewritten
as $x\in V$, $y^{\prime}\in S(x)$, so (IT) is equivalent to: For
any $\delta>0$, we can find neighborhoods $V$ of $\bar{x}$ and
$W$ of $\bar{y}$ such that \begin{equation}
\underbrace{x\in V}_{(1)},\underbrace{y^{\prime}\in S(x)}_{(2)},\underbrace{y\in W}_{(3)},\underbrace{y^{\prime}\in W}_{(4)}\mbox{ implies }\underbrace{x\in S^{-1}(y)+(T+\delta)(y^{\prime}-y)}_{(5)}.\label{eq:T-equiv-formula}\end{equation}
Next, take the set $A$ to be $A=\{y-y^{\prime}\}$. The condition
$y\in[S(x)+A]\cap W$ is equivalent to (2) and (3) combined. The definition
of $T(-\cdot)$-metric regularity can be written as: For any $\delta>0$,
we can find a neighborhood $V$ of $\bar{x}$ and $r>0$ such that
\begin{eqnarray}
 &  & x\in V\mbox{, }y^{\prime}\in S(x)\mbox{, }y\in\mathbb{B}(\bar{y},r)\mbox{ and }A\subset\mathbb{B}(\mathbf{0},r)\label{eq:T-MR-equiv-formula}\\
 & \mbox{ implies} & \underbrace{x\in S^{-1}(y)+(T+\delta)(-A)}_{(5)}.\nonumber \end{eqnarray}

We first show that (IT) implies (MR). Suppose \eqref{eq:T-equiv-formula}
holds for $W=\mathbb{B}(\bar{y},r^{\prime})$. If $r=\frac{r^{\prime}}{3}$,
then $y\in\mathbb{B}(\bar{y},r)$ and $A\subset\mathbb{B}(\mathbf{0},r)$
implies that $y^{\prime}\in\mathbb{B}(\bar{y},2r)=\mathbb{B}(\bar{y},\frac{2}{3}r^{\prime})\subset W$.
Therefore the formula in \eqref{eq:T-MR-equiv-formula} holds for
$r=\frac{r^{\prime}}{3}$.

Next, we show that (MR) implies (IT). Suppose \eqref{eq:T-MR-equiv-formula}
holds. If $y^{\prime}\in\mathbb{B}(\bar{y},\frac{r}{2})$ and $y\in\mathbb{B}(\bar{y},\frac{r}{2})$,
then $y-y^{\prime}\in\mathbb{B}(\mathbf{0},r)$, or $A\subset\mathbb{B}(\mathbf{0},r)$.
This means that \eqref{eq:T-equiv-formula} holds for $W=\mathbb{B}(\bar{y},\frac{r}{2})$.

For the equivalence of (OC) and (MR), note that (OC) is equivalent
to modifying the formula (5) in \eqref{eq:T-MR-equiv-formula} to
the equivalent condition $y\in S(x-(T+\delta)(y^{\prime}-y))$. 
\end{proof}

We now look at the connection between pseudo $T$-outer differentiability
and metric subregularity. For more on metric subregularity, we refer
the reader to \cite{DR04,DR09}. We define $T$-metric subregularity
as follows.
\begin{defn}
($T$-metric subregularity) \label{def:subregularity}Let $X$ and
$Y$ be Banach spaces, $S:X\rightrightarrows Y$ be a set-valued map,
and $T:Y\rightrightarrows X$ be positively homogeneous. $S$ is \emph{$T$-metrically
subregular} at $(\bar{x},\bar{y})\in\gph(S)$ if for any $\delta>0$,
there exists a neighborhood $V$ of $\bar{x}$ and $r>0$ such that,
for any $A\subset r\mathbb{B}$, (or equivalently, for any set $A\subset r\mathbb{B}$
containing exactly one element.)  \begin{equation}
x\in V,\bar{y}\in S(x)+A\mbox{ implies }x\in S^{-1}(\bar{y})+(T+\delta)(A).\label{eq:T-submetric-defn}\end{equation}

\end{defn}
Recall that metric subregularity of $S$ with modulus $\kappa$ at
$(\bar{x},\bar{y})$ can be written compactly as: For any $\kappa^{\prime}>\kappa$,
there exists a neighborhood $V$ of $\bar{x}$ such that \[
x\in V\mbox{ implies }d\big(x,S^{-1}(\bar{y})\big)\leq\kappa^{\prime}d\big(\bar{y},S(x)\big).\]
For the special case of $T(w):=\kappa|w|\mathbb{B}$, $T$-metric
subregularity is equivalent to metric subregularity with modulus $\kappa$.
Once again, an easy way to see this is to notice that both are equivalent
to condition (IT) below. This argument again makes use of the fact
that metric subregularity of $S$ at $(\bar{x},\bar{y})$ is equivalent
to the calmness of $S^{-1}$ at $\bar{y}$ for $\bar{x}$. 

Here is a result on the equivalences between $T$-metric subregularity
and pseudo outer $T$-differentiability.
\begin{thm}
(Generalized metric subregularity) Let $X$ and $Y$ be Banach spaces,
$S:X\rightrightarrows Y$ be a set-valued map, and $T:Y\rightrightarrows X$
be a positively homogeneous set-valued map. The following are equivalent:
\begin{enumerate}
\item [(MR${}^\prime$)]$S$ is $T(-\cdot)$-metric subregular at $(\bar{x},\bar{y})\in\gph(S)$.
\item [(IT${}^\prime$)]$S^{-1}$ is pseudo outer $T$-differentiable at
$\bar{y}$ for $\bar{x}$.
\end{enumerate}
\end{thm}
\begin{proof}
Condition (IT$^{\prime}$) can be written as: For any $\delta>0$,
there exists a neighborhood $V$ of $\bar{x}$ and $r>0$ such that
\begin{equation}
x\in V,x\in S^{-1}(y^{\prime}),y^{\prime}-\bar{y}\in r\mathbb{B}\mbox{ implies }x\in S^{-1}(\bar{y})+(T+\delta)(y^{\prime}-\bar{y}).\label{eq:T-sub-equiv-formula}\end{equation}
It is also clear that (MR$^{\prime}$) can be written in this form
with $A=\{\bar{y}-y^{\prime}\}$.
\end{proof}
Here is a lemma on pseudo strict $T$-differentiability amended from
\cite[Lemma 9.39]{RW98}, which is in turn attributed to \cite{Hen97}.
\begin{lem}
\label{lem:Extended-s_p_t_d}(Extended formulation of pseudo $T$-differentiability)
Consider a mapping $S:X\rightrightarrows Y$, a pair $(\bar{x},\bar{y})\in\gph(S)$,
a set $D\subset X$ containing $\bar{x}$, and a positively homogeneous
map $T:X\rightrightarrows Y$. Suppose that there is a $\delta>0$
such that \begin{equation}
T(w)\supset\delta|w|\mathbb{B}\mbox{ for all }w\in X.\label{eq:inner-bounded}\end{equation}

Then the following two conditions are equivalent:
\begin{enumerate}
\item [(a$^\prime$)]there exist neighborhoods $V$ of $\bar{x}$ and $W$
of $\bar{y}$ such that \[
S(x^{\prime})\cap W\subset S(\bar{x})+T(x^{\prime}-\bar{x})\mbox{ for all }x^{\prime}\in D\cap V.\]

\item [(b$^\prime$)]there exists a neighborhood $W$ of $\bar{y}$ such
that \[
S(x^{\prime})\cap W\subset S(\bar{x})+T(x^{\prime}-\bar{x})\mbox{ for all }x^{\prime}\in D.\]

\end{enumerate}
If in addition $|T|^{+}$ is finite, then the following two conditions
are equivalent as well:
\begin{enumerate}
\item [(a)]there exist neighborhoods $V$ of $\bar{x}$ and $W$ of $\bar{y}$
such that \[
S(x^{\prime})\cap W\subset S(x)+T(x^{\prime}-x)\mbox{ for all }x,x^{\prime}\in D\cap V.\]

\item [(b)]there exist neighborhoods $V$ of $\bar{x}$ and $W$ of $\bar{y}$
such that \[
S(x^{\prime})\cap W\subset S(x)+T(x^{\prime}-x)\mbox{ for all }x\in D\cap V,x^{\prime}\in D.\]

\end{enumerate}
\end{lem}
\begin{proof}
We first look at the equivalences of (a$^{\prime}$) and (b$^{\prime}$).
Trivially (b$^{\prime}$) implies (a$^{\prime}$), so assume that
(a$^{\prime}$) holds for neighborhoods $V=\mathbb{B}(\bar{x},\gamma)$
and $W=\mathbb{B}(\bar{y},\epsilon)$. We can reduce $\epsilon$ so
that $\epsilon\leq\delta\gamma$. Then $x^{\prime}\in D\backslash\mathbb{B}(\bar{x},\gamma)$
implies $W\subset S(\bar{x})+T(x^{\prime}-\bar{x})$, which gives
us what we seek.

We now look at the equivalences of (a) and (b). Let $\kappa:=|T|^{+}$.
Trivially (b) implies (a), so assume that (a) holds for neighborhoods
$V=\mathbb{B}(\bar{x},\gamma)$ and $W=\mathbb{B}(\bar{y},\epsilon)$.
We will verify that (b) holds for $V^{\prime}=\mathbb{B}(\bar{x},\gamma^{\prime})$
and $W^{\prime}=\mathbb{B}(\bar{y},\epsilon^{\prime})$ for some $0<\gamma^{\prime}<\gamma$,
$0<\epsilon^{\prime}<\epsilon$.

Fix any $x\in D\cap\mathbb{B}(\bar{x},\gamma^{\prime})$. Our assumption
gives us \[
S(x^{\prime})\cap\mathbb{B}(\bar{y},\epsilon^{\prime})\subset S(x)+T(x^{\prime}-x)\mbox{ when }x^{\prime}\in D\cap\mathbb{B}(\bar{x},\gamma),\]
and our goal is to demonstrate that this holds also when $x^{\prime}\in D\backslash\mathbb{B}(\bar{x},\gamma)$.
Note that $|x^{\prime}-x|>(\gamma-\gamma^{\prime})$. From applying
(a) to $x^{\prime}=\bar{x}$, we see that $\bar{y}\in S(x)+T(\bar{x}-x)$
and consequently $\mathbb{B}(\bar{y},\epsilon^{\prime})\subset S(x)+(\kappa\gamma^{\prime}+\epsilon^{\prime})\mathbb{B}$.
If $\kappa\gamma^{\prime}+\epsilon^{\prime}\leq\delta(\gamma-\gamma^{\prime})$,
then\begin{eqnarray*}
S(x^{\prime})\cap\mathbb{B}(\bar{y},\epsilon^{\prime}) & \subset & S(x)+(\kappa\gamma^{\prime}+\epsilon^{\prime})\mathbb{B}\\
 & \subset & S(x)+\delta(\gamma-\gamma^{\prime})\mathbb{B}\\
 & \subset & S(x)+\delta|x^{\prime}-x|\mathbb{B}\\
 & \subset & S(x)+T(x^{\prime}-x).\end{eqnarray*}
The condition $\kappa\gamma^{\prime}+\epsilon^{\prime}\leq\delta(\gamma-\gamma^{\prime})$
is easily achieved by making $\gamma^{\prime}$ and $\epsilon^{\prime}$
small enough, giving us the required conclusion.
\end{proof}
If $T^{-1}(\mathbf{0})=Y$, we have the following equivalent definitions
for $T$-metric regularity, $T$-open covering and $T$-metric subregularity
that do not require $A\subset r\mathbb{B}$.
\begin{prop}
\label{pro:Alternate-T-defn}(Alternate definition of metric (sub)regularity
and openness) Let $S:X\rightrightarrows Y$ and $T:Y\rightrightarrows X$
be set-valued maps, with $T$ positively homogeneous and $T^{-1}(\mathbf{0})=Y$
(or equivalently, $\mathbf{0}\in T(y)$ for all $y\in Y$). The following
is equivalent to the $T$-metric subregularity of $S$ at $(\bar{x},\bar{y})\in\gph(S)$:
\begin{enumerate}
\item [(MR$^\prime_1$)]For any $\delta>0$, there exists a neighborhood
$V$ of $\bar{x}$ such that, for any $x\in V$ and $A\subset Y$,
(or equivalently, for any set $A\subset Y$ containing exactly one
element) \[
\bar{y}\in S(x)+A\mbox{ implies }x\in S^{-1}(\bar{y})+(T+\delta)(A).\]

\end{enumerate}
Assume further that $|T|^{+}$ is finite. Then the following is equivalent
to the $T$-metric regularity of $S$ at $(\bar{x},\bar{y})\in\gph(S)$:
\begin{enumerate}
\item [(MR$_1$)]For any $\delta>0$, there exist neighborhoods $V$ of
$\bar{x}$ and $W$ of $\bar{y}$ such that, for any $x\in V$ and
$A\subset Y$, (or equivalently, for any set $A\subset Y$ containing
exactly one element), \[
y\in[S(x)+A]\cap W\mbox{ implies }x\in S^{-1}(y)+(T+\delta)(A).\]

\end{enumerate}
The following is equivalent to the $T$-open covering of $S$ at $(\bar{x},\bar{y})\in\gph(S)$:
\begin{enumerate}
\item [(OC$_1$)]For any $\delta>0$, there exist neighborhoods $V$ of
$\bar{x}$ and $W$ of $\bar{y}$ such that, for any $x\in V$ and
any set $A\subset Y$ (or equivalently, for any set $A\subset Y$
containing exactly one element),\[
[S(x)+A]\cap W\subset S\big(x+(T+\delta)(A)\big).\]

\end{enumerate}
\end{prop}
\begin{proof}
For the equivalence of $T$-metric subregularity and (MR$_{1}^{\prime}$),
the proof is similar, so we prove only the remaining equivalences.
The condition $T^{-1}(\mathbf{0})=Y$ ensures that for any $\delta>0$,
$(T+\delta)(y)\supset\delta|y|\mathbb{B}$ for all $y\in Y$. Furthermore,
$|T+\delta|^{+}$ is also finite for all finite $\delta>0$. So Lemma
\ref{lem:Extended-s_p_t_d} implies that pseudo strict  $T$-differentiability
of $S^{-1}$ at $\bar{y}$ for $\bar{x}$ is equivalent to: For any
$\delta>0$, there exists neighborhoods $V$ of $\bar{x}$ and $W$
of $\bar{y}$ such that \begin{equation}
\underbrace{x\in V}_{(1)},\underbrace{y^{\prime}\in S(x)}_{(2)},\underbrace{y\in W}_{(3)}\mbox{ implies }\underbrace{x\in S^{-1}(y)+(T+\delta)(y^{\prime}-y)}_{(5)}.\label{eq:new-T-equiv}\end{equation}
The difference here from \eqref{eq:T-equiv-formula} is that the condition
$y^{\prime}\in W$ is superfluous. We also note that $T(-\cdot)$-metric
regularity is equivalent to rewriting (2) and (3) as $y\in[S(x)+(y-y^{\prime})]\cap W$
in the above, while $-T(-\cdot)$-open covering is rewriting (2) and
(3) as $y\in[S(x)+(y-y^{\prime})]\cap W$ and (5) as $y\in S(x-(T+\delta)(y^{\prime}-y))$.
Combined with Theorem \ref{thm:T-equivalences}, this proves the alternative
definitions of $T$-metric regularity and $T$-open covering in (MR$_{1}$)
and (OC$_{1}$). 
\end{proof}
To close this section, we shall illustrate in Proposition \ref{pro:Reduce-(C,T)}
how $T$-metric regularity can be a more precise tool than metric
regularity. Recall metric regularity with modulus $\kappa$ is often
written compactly as: For all $\kappa^{\prime}>\kappa$, there exists
neighborhoods $V$ of $\bar{x}$ and $W$ of $\bar{y}$ such that
\begin{equation}
x\in V,y\in W\mbox{ implies }d\big(x,S^{-1}(y)\big)\leq\kappa^{\prime}d\big(y,S(x)\big).\label{eq:MR-form}\end{equation}
Another way of writing \eqref{eq:MR-form} is \[
x\in V,y\in W,t>0\mbox{ and }d\big(y,S(x)\big)\leq t\mbox{ implies }d\big(x,S^{-1}(y)\big)\leq\kappa^{\prime}t,\]
or equivalently,\begin{equation}
x\in V,y\in W,t>0\mbox{ and }y\in S(x)+t\mathbb{B}\mbox{ implies }x\in S^{-1}(y)+\kappa^{\prime}t\mathbb{B}.\label{eq:MR-form2}\end{equation}
If \eqref{eq:MR-form2} holds for $V=\mathbb{B}(\bar{x},r_{1})$ and
$W=\mathbb{B}(\bar{y},r_{2})$, then for any $x\in V$ and $y\in W$,
\begin{eqnarray*}
 &  & y\in\bar{y}+r_{2}\mathbb{B}\\
 & \implies & y\in S(\bar{x})+r_{2}\mathbb{B}\\
 & \implies & \bar{x}\in S^{-1}(y)+\kappa^{\prime}r_{2}\mathbb{B}\\
 & \implies & x\in S^{-1}(y)+(\kappa^{\prime}r_{2}+r_{1})\mathbb{B}.\end{eqnarray*}
This means that $x\in V$, $y\in W$ and $t\geq r_{2}+\frac{r_{1}}{\kappa^{\prime}}$
implies $x\in S^{-1}(y)+\kappa^{\prime}t\mathbb{B}$, so amending
the condition $t>0$ in \eqref{eq:MR-form2} to $0<t\leq r_{2}+\frac{r_{1}}{\kappa^{\prime}}$
does not change the statement there. Motivated by the above, we define
$(C,T)$-metric regularity below, and show that this concept is equivalent
to $T^{\prime}$-metric regularity for some appropriately defined
$T^{\prime}:Y\rightrightarrows X$. 
\begin{defn}
($(C,T)$-regularity) Let $X$ and $Y$ be Banach spaces, $S:X\rightrightarrows Y$
be a set-valued map,  $T:Y\rightrightarrows X$ be positively homogeneous,
and $C\subset Y$ be closed. We say that $S$ is \emph{$(C,T)$-metrically
regular} at $(\bar{x},\bar{y})\in\gph(S)$ if for any $\delta>0$,
there exist neighborhoods $V$ of $\bar{x}$ and $W$ of $\bar{y}$
and $r>0$ such that \[
x\in V,y\in[S(x)+tC]\cap W,t>0\mbox{ and }tC\subset r\mathbb{B}\mbox{ implies }x\in S^{-1}(y)+(T+\delta)(tC).\]

\end{defn}
It follows from our earlier discussion that $(\mathbb{B},T)$-metric
regularity, where $T(w):=\kappa|w|\mathbb{B}$, is equivalent to metric
regularity with modulus $\kappa$. If $C$ is chosen to be a set different
from the unit ball $\mathbb{B}$, then $(C,T)$-metric regularity
can identify sensitive and less sensitive directions. For example
one can choose $C$ as the unit ball under a different norm. We now
prove the equivalence of $(C,T)$-metric regularity and $T^{\prime}$-metric
regularity for some $T^{\prime}:Y\rightrightarrows X$. 
\begin{prop}
\label{pro:Reduce-(C,T)} (Reduction of $(C,T)$-metric regularity
to $T$-metric regularity) Let $X$ and $Y$ be Banach spaces, $S:X\rightrightarrows Y$
be a set-valued map, and $T:Y\rightrightarrows X$ be positively homogeneous,
$C$ is convex, and $\epsilon\mathbb{B}\subset C\subset R\mathbb{B}$
for some $\epsilon,R>0$. Then $S$ is $T^{\prime}$-metrically regular
at $(\bar{x},\bar{y})$ if and only if $S$ is $(C,T)$-metrically
regular there, where $T^{\prime}:Y\rightrightarrows X$ is defined
by $T^{\prime}(w)=T(tC)$ for $t=\min\{\lambda\mid w\in\lambda C\}$. \end{prop}
\begin{proof}
We recall from the proof of Theorem \ref{thm:T-equivalences} that
$T^{\prime}$-metric regularity is equivalent to: For all $\delta>0$,
there is a neighborhood $V$ of $\bar{x}$ and $r>0$ such that \begin{eqnarray}
 &  & x\in V,y^{\prime}\in S(x),y\in\mathbb{B}(\bar{y},r),y-y^{\prime}\in r\mathbb{B}\label{eq:TPMR}\\
 & \mbox{ implies } & x\in S^{-1}(y)+(T^{\prime}+\delta)(y-y^{\prime}).\nonumber \end{eqnarray}
Next, we note that $(C,T)$-metric regularity is equivalent to: For
any $\delta>0$, there is a neighborhood $V$ of $\bar{x}$ and $r>0$
such that\[
x\in V,y\in[S(x)+tC]\cap\mathbb{B}(\bar{y},r),tC\subset r\mathbb{B}\mbox{ implies }x\in S^{-1}(y)+(T+\delta)(tC),\]
which is in turn equivalent to: For any $\delta>0$, there is a neighborhood
$V$ of $\bar{x}$ and $r>0$ such that\begin{eqnarray}
 &  & x\in V,y^{\prime}\in S(x),y\in\mathbb{B}(\bar{y},r),y-y^{\prime}\in tC,tC\subset r\mathbb{B}\label{eq:CTMR}\\
 & \mbox{implies } & x\in S^{-1}(y)+(T+\delta)(tC).\nonumber \end{eqnarray}
Suppose $S$ is $T^{\prime}$-metrically regular at $(\bar{x},\bar{y})$.
Let $\delta,r>0$ and $V$ be such that \eqref{eq:TPMR} holds, and
suppose $x\in V$, $y^{\prime}\in S(x)$ and $y\in\mathbb{B}(\bar{y},r)$.
Let $t>0$ be such that $y-y^{\prime}\in tC$. Then we have $x\in S^{-1}(y)+(T^{\prime}+\delta)(y-y^{\prime})$.
Let $t^{\prime}=\min\{\lambda\mid y-y^{\prime}\in\lambda C\}$. Then
\begin{eqnarray*}
T^{\prime}(y-y^{\prime}) & = & T(t^{\prime}C),\\
\mbox{ so }(T^{\prime}+\delta)(y-y^{\prime}) & \subset & (T+\delta)(t^{\prime}C),\end{eqnarray*}
and $t^{\prime}\leq t$. Since $C$ is convex and $0\in\mbox{int}(C)$,
$t^{\prime}C\subset tC$, so $(T+\delta)(t^{\prime}C)\subset(T+\delta)(tC)$,
which implies $x\in S^{-1}(y)+(T+\delta)(tC)$. This in turn means
that \eqref{eq:CTMR} holds, so $S$ is $(C,T)$-metrically regular
at $(\bar{x},\bar{y})$.

For the converse, suppose that $S$ is $(C,T)$-metrically regular
at $(\bar{x},\bar{y})$. That is, for any $\delta>0$, there is a
neighborhood $V$ of $\bar{x}$ and $r>0$ such that \eqref{eq:CTMR}
holds. Recall $\epsilon\mathbb{B}\subset C\subset R\mathbb{B}$. If
$y-y^{\prime}\in\frac{r\epsilon}{R}\mathbb{B}$, then the minimum
$t>0$ such that $y-y^{\prime}\in tC$ gives $tC\subset r\mathbb{B}$,
so the condition $tC\subset r\mathbb{B}$ is superfluous in \eqref{eq:CTMR}.
The minimality of $t$ ensures $T(tC)=T^{\prime}(y-y^{\prime})$.
The minimality of $t$ and $\epsilon\mathbb{B}\subset C$ also ensures
$\epsilon\leq\frac{1}{t}|y-y^{\prime}|$, so we have $t\epsilon\leq|y-y^{\prime}|$.
The fact that $\frac{1}{t}(y-y^{\prime})\in C\subset R\mathbb{B}$
gives $\frac{1}{t}|y-y^{\prime}|\leq R$. Then \begin{eqnarray*}
(T+\delta)(tC) & = & T(tC)+\delta(tC)\\
 & \subset & T^{\prime}(y-y^{\prime})+\delta tR\mathbb{B}\\
 & = & T^{\prime}(y-y^{\prime})+\delta\frac{R}{\epsilon}t\epsilon\mathbb{B}\\
 & \subset & T^{\prime}(y-y^{\prime})+\delta\frac{R}{\epsilon}|y-y^{\prime}|\mathbb{B}\\
 & = & \Big(T^{\prime}+\delta\frac{R}{\epsilon}\Big)(y-y^{\prime}).\end{eqnarray*}
To conclude, we have \begin{eqnarray*}
 &  & x\in V,y^{\prime}\in S(x),y\in\mathbb{B}(\bar{y},r),y-y^{\prime}\in\frac{r\epsilon}{R}\mathbb{B}\\
 & \mbox{implies } & x\in S^{-1}(y)+\Big(T^{\prime}+\delta\frac{R}{\epsilon}\Big)(y-y^{\prime}).\end{eqnarray*}
The form in the last expression is similar to that of \eqref{eq:TPMR}.
We easily deduce that $S$ is $T^{\prime}$-metrically regular at
$(\bar{x},\bar{y})$ as needed.
\end{proof}
As a corollary, we have the equivalence of $T$-metric regularity
with $T(w):=\kappa|w|\mathbb{B}$ and metric regularity with modulus
$\kappa$.

\section{\label{sec:Strict--T-from-T}Strict $T$-differentiability from outer
 $T$-differentiability}

Suppose $S:X\rightrightarrows Y$ is such that $\bar{y}\in S(\bar{x})$,
$T:X\rightrightarrows Y$ is positively homogeneous, and $S:X\rightrightarrows Y$
is pseudo strictly $T$-differentiable at $\bar{x}$ for $\bar{y}$.
Then for any $\delta>0$, there are neighborhoods $U$ of $\bar{x}$
and $V$ of $\bar{y}$ such that $S$ is pseudo outer $(T+\delta)$-differentiable
at $x$ for $y$ for all $x\in U$ and $y\in V\cap S(x)$. The main
result in this section, Theorem \ref{thm:strict-T-from-T}, is to
show that the converse holds with additional assumptions.

We now study the relationship between pseudo outer $T$-differentiability
and pseudo strict $T$-differentiability. The relation between pseudo
(outer/ strict) $T$-differentiability and (outer/ strict) $T$-differentiability
can be obtained from Theorem \ref{thm:T-diff-from-pseudo}. First,
here is a lemma on pseudo outer $T$-dif\-fer\-en\-tia\-ble functions
on a convex set that is comparable to the second part of \cite[Theorem 9.2]{RW98}.
This is extended from results in \cite{Li94,Rob07}.
\begin{lem}
\label{lem:local_Aubin_1}(Pseudo outer $T$-differentiability) Let
$D\subset X$ be a convex set, and $S:D\rightrightarrows\mathbb{R}^{n}$
be a closed-valued, osc set-valued map satisfying $\bar{y}\in S(\bar{x})$.
Assume that $T:X\rightrightarrows\mathbb{R}^{n}$ is a positively
homogeneous closed convex valued set-valued map and $|T|^{+}\leq\kappa$
for some $\kappa>0$. Suppose that $r>0$ is such that 
\begin{enumerate}
\item $\liminf_{x^{\prime}\xrightarrow[D]{}x}S(x^{\prime})\supset S(x)\cap\mathbb{B}(\bar{y},r)$
with respect to $D$ for all $x\in D$ (This is true when $S$ is
inner semicontinuous.), and
\item $S$ is pseudo outer $T$-differentiable for all $x\in D$ and $y\in S(x)\cap\mathbb{B}(\bar{y},r)$, 
\end{enumerate}
Then for any $\delta>0$, $x_{0},x_{1}\in D${\Large{} }and $r^{\prime}$
satisfying $(\kappa+\delta)|x_{0}-x_{1}|+r^{\prime}<r$, we have $S(x_{1})\cap\mathbb{B}(\bar{y},r^{\prime})\subset S(x_{0})+(T+\delta)(x_{1}-x_{0})$.\end{lem}
\begin{proof}
By the definition of pseudo $T$-differentiability, for any $\delta>0$,
$x\in D$ and $y\in S(x)\cap\mathbb{B}(\bar{y},r)$, there are neighborhoods
$V_{(x,y)}$ of $x$ and $U_{(x,y)}$ of $y$ such that\[
S(x^{\prime})\cap U_{(x,y)}\subset S(x)+(T+\delta)(x^{\prime}-x)\mbox{ for all }x^{\prime}\in D\cap V_{(x,y)}.\]
Since $S(x)\cap\mathbb{B}(\bar{y},r)$ is compact, we may choose $y_{1},\dots,y_{k}\in\mathbb{B}(\bar{y},r)$
such that \[
S(x)\cap\mathbb{B}(\bar{y},r)\subset\cup_{i=1}^{k}U_{(x,y_{i})}.\]
Let $U_{x}$ be the right hand side of the above formula. Clearly,
$U_{x}\cup[\mathbb{B}(\bar{y},r)]^{c}$ is an open set containing
$S(x)$, where $[\mathbb{B}(\bar{y},r)]^{c}$ is the complement of
$\mathbb{B}(\bar{y},r)$. Since $S$ is outer semicontinuous, this
implies that there is some neighborhood $V_{x}\subset\cap_{i=1}^{k}V_{(x,y_{i})}$
such that $S(x^{\prime})\subset U_{x}\cup[\mathbb{B}(\bar{y},r)]^{c}$
for all $x^{\prime}\in V_{x}$, which implies that \[
S(x^{\prime})\cap\mathbb{B}(\bar{y},r)\subset U_{x}\mbox{ for all }x^{\prime}\in V_{x}.\]
This gives us \begin{eqnarray}
S(x^{\prime})\cap\mathbb{B}(\bar{y},r) & \subset & S(x^{\prime})\cap U_{x}\label{eq:(3)-better}\\
 & \subset & \cup_{i=1}^{k}[S(x^{\prime})\cap U_{(x,y_{i})}]\nonumber \\
 & \subset & S(x)+(T+\delta)(x^{\prime}-x)\mbox{ for all }x^{\prime}\in D\cap V_{x}.\nonumber \end{eqnarray}
Pick $x_{0},x_{1}\in D$. For each $t\in(0,1)$, let $x_{t}=(1-t)x_{0}+tx_{1}$.
Formula \eqref{eq:(3)-better} ensures that for each $t\in[0,1]$,
there is a ball $\mathbb{B}(x_{t},\rho_{t})$ such that for each $x^{\prime}$
in $D\cap\mathbb{B}(x_{t},\rho_{t})$, $S(x^{\prime})\cap\mathring{\mathbb{B}}(\bar{y},r)\subset S(x_{t})+(T+\delta)(x^{\prime}-x_{t})$.
Here, $\mathring{\mathbb{B}}$ denotes an open ball. Define\begin{eqnarray*}
\tau & := & \sup\{t\in[0,1]\mid\mbox{for each }s\in[0,t],\\
 &  & \quad\quad S(x_{s})\cap\mathring{\mathbb{B}}(\bar{y},r-(\kappa+\delta)|x_{s}-x_{0}|)\subset S(x_{0})+(T+\delta)(x_{s}-x_{0})\}.\end{eqnarray*}
We have $\tau>0$ because $\rho_{0}$ is positive. We show first
that \begin{equation}
S(x_{\tau})\cap\mathring{\mathbb{B}}(\bar{y},r-(\kappa+\delta)|x_{\tau}-x_{0}|)\subset S(x_{0})+(T+\delta)(x_{\tau}-x_{0}).\label{eq:robinson-5.1}\end{equation}
By assumption the set $S(x_{0})$ is closed. The set $(T+\delta)(x_{\tau}-x_{0})$
is closed and bounded in $\mathbb{R}^{n}$, and hence compact. So
$S(x_{0})+(T+\delta)(x_{\tau}-x_{0})$ is closed as well; let $Q$
be the complement of the RHS in \eqref{eq:robinson-5.1}. If \eqref{eq:robinson-5.1}
were not true, then $S(x_{\tau})$ would meet the open set $Q\cap\mathring{\mathbb{B}}(\bar{y},r-(\kappa+\delta)|x_{\tau}-x_{0}|)$,
at $\tilde{y}$ say. If $x_{0},x_{1}$ were close enough to $\bar{x}$,
then $\tilde{y}\in S(x_{\tau})\cap\mathbb{B}(\bar{y},r)\subset\liminf_{x^{\prime}\rightarrow x_{\tau}}S(x^{\prime})$.
Therefore, for any $\tilde{r}>0$, there exists some $\epsilon>0$
such that \[
|x^{\prime}-x_{\tau}|<\epsilon\mbox{ implies }S(x^{\prime})\cap\mathbb{B}(\tilde{y},\tilde{r})\neq\emptyset.\]
If $\sigma\in(0,\tau)$ is such that $|x_{\sigma}-x_{\tau}|<\epsilon$,
then $S(x_{\sigma})\cap\mathbb{B}(\tilde{y},\tilde{r})\neq\emptyset$.
If $\tilde{r}$ is such that $\mathbb{B}(\tilde{y},\tilde{r})\subset\mathring{\mathbb{B}}(\bar{y},r-(\kappa+\delta)|x_{\tau}-x_{0}|)$,
then \[
[S(x_{0})+(T+\delta)(x_{\sigma}-x_{0})]\cap\mathbb{B}(\tilde{y},\tilde{r})\neq\emptyset.\]
We can find $\tilde{y}_{\sigma}$ such that $|\tilde{y}-\tilde{y}_{\sigma}|<\tilde{r}$,
and $\tilde{y}_{\sigma}\in S(x_{0})+(T+\delta)(x_{\sigma}-x_{0})$.
Next, recall that $S(x_{0})+(T+\delta)(x_{\tau}-x_{0})=S(x_{0})+(T+\delta)(x_{\tau}-x_{\sigma})+(T+\delta)(x_{\sigma}-x_{0})$.
This means that there exists $y_{\sigma}\in S(x_{0})+(T+\delta)(x_{\tau}-x_{0})$
for which $|\tilde{y}_{\sigma}-y_{\sigma}|\leq(\kappa+\delta)|x_{\tau}-x_{\sigma}|$.
Combining the two gives \[
d\big(\tilde{y},S(x_{0})+(T+\delta)(x_{\tau}-x_{0})\big)\leq|\tilde{y}-\tilde{y}_{\sigma}|+|\tilde{y}_{\sigma}-y_{\sigma}|<\tilde{r}+(\kappa+\delta)|x_{\tau}-x_{\sigma}|.\]
Since the sum on the right hand side can be made arbitrarily small,
we have $\tilde{y}\in S(x_{0})+(T+\delta)(x_{\tau}-x_{0})$, which
contradicts $\tilde{y}\in Q$. This establishes \eqref{eq:robinson-5.1}.

If $\tau$ were less than $1$ there would be $\lambda\in(\tau,1)$
with $|x_{\lambda}-x_{\tau}|<\rho_{\tau}$, such that \begin{equation}
S(x_{\lambda})\cap\mathring{\mathbb{B}}(\bar{y},r-(\kappa+\delta)|x_{\lambda}-x_{0}|)\not\subset S(x_{0})+(T+\delta)(x_{\lambda}-x_{0}).\label{eq:robinson-5.2}\end{equation}
However, we would then have from the definition of $\rho_{\tau}$
that \[
S(x_{\lambda})\cap\mathring{\mathbb{B}}(\bar{y},r-(\kappa+\delta)|x_{\lambda}-x_{0}|)\subset S(x_{\tau})+(T+\delta)(x_{\lambda}-x_{\tau}),\]
 and so\begin{eqnarray*}
 &  & S(x_{\lambda})\cap\mathring{\mathbb{B}}(\bar{y},r-(\kappa+\delta)|x_{\lambda}-x_{0}|)\\
 & \subset & [S(x_{\tau})+(T+\delta)(x_{\lambda}-x_{\tau})]\cap\mathring{\mathbb{B}}(\bar{y},r-(\kappa+\delta)|x_{\lambda}-x_{0}|)\\
 & = & \big([S(x_{\tau})\cap\mathring{\mathbb{B}}(\bar{y},r-(\kappa+\delta)|x_{\tau}-x_{0}|)]+(T+\delta)(x_{\lambda}-x_{\tau})\big)\\
 &  & \quad\quad\cap\mathring{\mathbb{B}}(\bar{y},r-(\kappa+\delta)|x_{\lambda}-x_{0}|)\\
 & \subset & \big([S(x_{0})+(T+\delta)(x_{\tau}-x_{0})]+(T+\delta)(x_{\lambda}-x_{\tau})\big)\\
 &  & \quad\quad\cap\mathring{\mathbb{B}}(\bar{y},r-(\kappa+\delta)|x_{\lambda}-x_{0}|)\\
 & \subset & S(x_{0})+(T+\delta)(x_{\lambda}-x_{0}),\end{eqnarray*}
where the inclusion in the last line holds because $x_{0}$, $x_{\tau}$
and $x_{\lambda}$ are collinear and $T$ is a positively homogeneous
convex valued map. This contradicts \eqref{eq:robinson-5.2}, so $\tau$
must be $1$. Putting $\tau=1$ in \eqref{eq:robinson-5.1} shows
that $S(x_{1})\cap\mathring{\mathbb{B}}(\bar{y},r-(\kappa+\delta)|x_{1}-x_{0}|)\subset S(x_{0})+(T+\delta)(x_{1}-x_{0})$.
 The required conclusion follows immediately.
\end{proof}
By decreasing $r$ and/or the size of $D$, we deduce that $S$ is
pseudo strictly $T$-differentiable at $\bar{x}$ for $\bar{u}$.
This is summarized in the theorem below. 
\begin{thm}
\label{thm:strict-T-from-T}(Pseudo strict $T$-differentiability
from pseudo outer $T$-dif\-fer\-en\-ti\-a\-bil\-i\-ty) Suppose that
$S:D\rightrightarrows\mathbb{R}^{n}$ is closed-valued, osc and $\bar{y}\in S(\bar{x})$,
where $D=\dom(S)\subset X$ is convex. Let $T:X\rightrightarrows\mathbb{R}^{n}$
be a positively homogeneous closed convex valued set-valued map such
that $|T|^{+}\leq\kappa$ for some $\kappa>0$. If for any $\delta>0$,
there exists open convex sets $U$ of $\bar{x}$ and $V$ of $\bar{y}$
such that 
\begin{enumerate}
\item $S$ is pseudo outer  $(T+\delta)$-differentiable at $x$ for $y$
whenever $x\in U$, $y\in V$ and $y\in S(x)$, 
\item $\liminf_{x^{\prime}\xrightarrow[D]{}x}S(x^{\prime})\supset S(x)\cap V$
for all $x\in U$. (This is true when $S$ is inner semicontinuous.)
\end{enumerate}
then $S$ is pseudo strictly $T$-differentiable at $\bar{x}$ for
$\bar{y}$. Furthermore, the conclusion still holds if we weaken condition
(1) to:
\begin{enumerate}
\item [(1')]$S$ is pseudo outer $(T+\delta)$-differentiable at $x$ for
$y$ whenever $x\in U$, $y\in V$ and $y\in S(x)$ but $(x,y)\neq(\bar{x},\bar{y})$.
\end{enumerate}
\end{thm}
\begin{proof}
Choose any $\delta>0$. There are neighborhoods $U$ containing $\bar{x}$
and $\mathbb{B}(\bar{y},r)$ containing $\bar{y}$ such that $S$
is pseudo outer $(T+\delta)$-differentiable at $x$ for $y$ for
all $x\in U$, $y\in\mathbb{B}(\bar{y},r)$. We can reduce the size
of $U$ so that the diameter of $U$, say $d$, satisfies $r-(\kappa+2\delta)d>0$.
Let $r^{\prime}=\frac{1}{2}[r-(\kappa+2\delta)d]$. By Lemma \ref{lem:local_Aubin_1},
\[
S(x_{1})\cap\mathbb{B}(\bar{y},r^{\prime})\subset S(x_{0})+(T+2\delta)(x_{1}-x_{0})\mbox{ for all }x_{0},x_{1}\in U\cap D.\]
Since $\delta$ is arbitrary, $S$ is pseudo strictly $T$-differentiable
at $\bar{x}$ for $\bar{y}$ as needed.

To prove the second part, we only need to prove that if condition
(1') holds, then $S$ is pseudo outer $(T+2\delta)$-differentiable
at $\bar{x}$ for $\bar{y}$. Again, suppose we have neighborhoods
$U$ of $\bar{x}$ and $V=\mathbb{B}(\bar{y},r)$ of $\bar{y}$ respectively
such that condition (1') holds. First, we prove that for all $x\in\mathbb{B}(\bar{x},\frac{r}{2(\kappa+2\delta)})\cap U\cap D$,
\begin{equation}
S(x)\cap\mathbb{B}\left(\bar{y},\frac{r}{2}\right)\subset S(\bar{x})+(T+2\delta)(x-\bar{x}).\label{eq:calm_1}\end{equation}
Since $S$ is outer semicontinuous at $\bar{x}$, for any $\epsilon>0$,
there exists a convex combination of $\{x,\bar{x}\}$ arbitrarily
close enough to $\bar{x}$, say $\hat{x}$, such that \[
S(\hat{x})\cap\mathbb{B}\left(\bar{y},\frac{r}{2}+(\kappa+2\delta)|x-\bar{x}|\right)\subset S(\bar{x})+\epsilon\mathbb{B}.\]
Choose the domain $D^{\prime}$ to be $\mathbb{B}(\frac{1}{2}(\hat{x}+x),\frac{1}{2}|\hat{x}-x|)$.
Both $x,\hat{x}$ are in $D^{\prime}$, which is convex, and $\frac{r}{2}+(\kappa+2\delta)|x-\hat{x}|<r$,
so the conditions for Lemma \ref{lem:local_Aubin_1} are satisfied,
and we have \[
S(x)\cap\mathbb{B}\left(\bar{y},\frac{r}{2}\right)\subset S(\hat{x})+(T+2\delta)(x-\hat{x}).\]
Then \begin{eqnarray}
S(x)\cap\mathbb{B}\left(\bar{y},\frac{r}{2}\right) & \subset & [S(\hat{x})+(T+2\delta)(x-\hat{x})]\cap\mathbb{B}\left(\bar{y},\frac{r}{2}\right)\nonumber \\
 & = & \left(\left[S(\hat{x})\cap\mathbb{B}\left(\bar{y},\frac{r}{2}+(\kappa+2\delta)|x-\bar{x}|\right)\right]+(T+2\delta)(x-\hat{x})\right)\cap\mathbb{B}\left(\bar{y},\frac{r}{2}\right)\nonumber \\
 & \subset & \left[S(\hat{x})\cap\mathbb{B}\left(\bar{y},\frac{r}{2}+(\kappa+2\delta)|x-\bar{x}|\right)\right]+(T+2\delta)(x-\hat{x})\nonumber \\
 & \subset & [S(\bar{x})+\epsilon\mathbb{B}]+(T+2\delta)(x-\hat{x})\nonumber \\
 & = & S(\bar{x})+(T+2\delta)(x-\hat{x})+\epsilon\mathbb{B}.\label{eq:sec-8-fin1}\end{eqnarray}
To continue in proving \eqref{eq:calm_1}, we need to show that \begin{equation}
(T+2\delta)(x-\hat{x})\subset(T+2\delta)(x-\bar{x})+(\kappa+2\delta)|\hat{x}-\bar{x}|\mathbb{B}.\label{eq:sec-8-fin2}\end{equation}
Since $x$, $\hat{x}$ and $\bar{x}$ are collinear, we can write
$\hat{x}-\bar{x}$ as $\lambda(x-\bar{x})$, where $0<\lambda<1$.
Suppose $w\in(T+2\delta)(x-\hat{x})$, or $\frac{1}{1-\lambda}w\in(T+2\delta)(x-\bar{x})$.
Then \begin{eqnarray*}
w & = & \frac{1}{1-\lambda}w-\frac{\lambda}{1-\lambda}w\\
 & \in & (T+2\delta)(x-\bar{x})-\lambda(T+2\delta)(x-\bar{x})\\
 & = & (T+2\delta)(x-\bar{x})-(T+2\delta)(\hat{x}-\bar{x})\\
 & \subset & (T+2\delta)(x-\bar{x})+(\kappa+2\delta)|\hat{x}-\bar{x}|\mathbb{B}.\end{eqnarray*}
So \eqref{eq:sec-8-fin2} holds.

Since $\epsilon$ and $|\hat{x}-\bar{x}|$ can be made arbitrarily
small, we have, from \eqref{eq:sec-8-fin1} and \eqref{eq:sec-8-fin2},
$S(x)\cap\mathbb{B}(\bar{y},\frac{r}{2})\subset S(\bar{x})+(T+2\delta)(x-\bar{x})$
as claimed. As $x$ is arbitrary in $\mathbb{B}(\bar{x},\frac{r}{2(\kappa+2\delta)})\cap U\cap D$,
this means that $S$ is pseudo outer $(T+2\delta)$-differentiable
at $\bar{x}$ for $\bar{y}$, and we are done.
\end{proof}
The corollary below addresses calmness and Lipschitz continuity. We
did not explicitly treat the case where either the Lipschitz or calmness
moduli could be infinity, but this is still easy.
\begin{cor}
\label{cor:lip=00003Dlimsup_calm}(Calmness and Lipschitz moduli)
Suppose that $S:D\rightrightarrows\mathbb{R}^{n}$ is closed valued,
osc with $D=\dom(S)\subset X$ and $\bar{y}\in S(\bar{x})$. We have
\[
\lip\, S(\bar{x}\mid\bar{y})\geq\limsup_{(x,y)\xrightarrow[\scriptsize\gph(S)]{}(\bar{x},\bar{y})}\calm\, S(x\mid y).\]
If $D$ is locally convex at $\bar{x}$, and there exist neighborhoods
$V$ of $\bar{x}$ and $W$ of $\bar{y}$ such that $\liminf_{x^{\prime}\xrightarrow[X]{}x}S(x^{\prime})\supset S(x)\cap W$
for all $x\in V$ (which is the case when $S$ is inner semicontinuous),
then equality holds. In addition, we also have \[
\lip\, S(\bar{x}\mid\bar{y})=\limsup_{{(x,y)\xrightarrow[\scriptsize\gph(S)]{}(\bar{x},\bar{y})\atop (x,y)\neq(\bar{x},\bar{y})}}\calm\, S(x\mid y).\]

\end{cor}
In the single-valued case, we have the following corollary.
\begin{cor}
\label{cor:Single-valued-T-diff}(Single-valued functions) Let $f:D\rightarrow\mathbb{R}^{n}$
be continuous, where $D\subset X$ is convex. 
\begin{enumerate}
\item Let $T:X\rightrightarrows\mathbb{R}^{n}$ be a closed convex valued
positively homogeneous map such that $|T|^{+}$ is finite. The function
$f$ is strictly $T$-differentiable at $\bar{x}$ if and only if
for all $\delta>0$, there is a convex neighborhood $U$ of $\bar{x}$
such that $f$ is $(T+\delta)$-differentiable at all points in $U\cap D$.
\item For any $\bar{x}\in D$, \[
\lip f(\bar{x})=\limsup_{{x\xrightarrow[D]{}\bar{x}\atop x\neq\bar{x}}}\calm\, f(x).\]

\end{enumerate}
\end{cor}
\begin{acknowledgement*}
I thank the associate editor and the two anonymous referees for their
careful reading and insightful comments which have made the paper
much better than it was. I also thank Asen Dontchev, Diethard Klatte,
Bernd Kummer, Adrian Lewis and Boris Mordukhovich for comments on
a previous version of this paper which have led to improvements, and
on the bibliography on the subject. In particular, I thank Lionel
Thibault for bringing to my attention the papers \cite{Ioffe79,Iof81,Thi82},
and to Alexander Kruger for several suggestions. Much of this paper
was written while I was in the Fields Institute in Toronto, which
has provided a wonderful environment for working on this paper.\end{acknowledgement*}

\end{document}